\documentclass[11pt, reqno]{amsart}    

\usepackage{amsmath,amssymb,latexsym}
\usepackage{graphicx,xcolor,mathrsfs}
\usepackage{cite}
\usepackage[font={scriptsize}]{caption}

\numberwithin{equation}{section}

\newcommand{\bigO}{\mathcal{O}}

\newcommand{\abs}[1]{\left\lvert#1\right\rvert}
\newcommand\norm[1]{\left\lvert#1\right\rvert}

\newcommand{\R}{\mathbb{R}}
\newcommand{\Z}{\mathbb{Z}}
\newcommand{\N}{\mathbb{N}}

\newcommand{\ee}{\mathrm{e}}
\newcommand{\I}{\mathrm{i}}
\newcommand{\ii}{\mathrm{i}}

\newcommand{\dk}{\, \mathrm{d}k}

\newcommand{\ds}{\, \mathrm{d}s}

\newcommand{\dx}{\, \mathrm{d}x}
\newcommand{\dy}{\, \mathrm{d}y}
\newcommand{\dz}{\, \mathrm{d}z}
\newcommand{\diff}{\mathrm{d}}
\newcommand{\Diff}{\mathrm{D}}

\newcommand{\II}{{\mathcal I}}

\newcommand{\NN}{{\mathcal N}}

\newcommand{\Schwartz}{{\mathcal S}}

\DeclareMathOperator{\re}{Re}

\newcommand{\supp}{\mathop{\mathrm{supp}}}

\newtheorem{theorem}{Theorem}[section]

\newtheorem{lemma}[theorem]{Lemma}
\newtheorem{proposition}[theorem]{Proposition}
\newtheorem{corollary}[theorem]{Corollary}
\newtheorem{remark}[theorem]{Remark}

\theoremstyle{definition}

\usepackage[colorlinks=true]{hyperref}
\hypersetup{urlcolor=blue, citecolor=red, linkcolor=blue}

\begin{document}

\allowdisplaybreaks

\title[Existence of DS type solitary waves for the FDKP equation]{Existence of Davey--Stewartson type solitary waves for the fully dispersive Kadomtsev--Petviashvilii equation}

\author{Mats Ehrnstr\"om}
\address{Department of Mathematical Sciences, Norwegian University of Science and Technology, 7491 Trondheim, Norway}
\email{mats.ehrnstrom@ntnu.no}

\author{Mark D. Groves}
\address{Fachrichtung Mathematik, Universit\"at des Saarlandes, Postfach 151150, 66041 Saarbr\"ucken, Germany}
\email{groves@math.uni-sb.de}

\author{Dag Nilsson}
\address{Fachrichtung Mathematik, Universit\"at des Saarlandes, Postfach 151150, 66041 Saarbr\"ucken, Germany}
\email{nilsson@math.uni-sb.de}

\thanks{ME and DN acknowledge the support by grant no. 250070 from the Research Council of Norway.}

\subjclass[2010]{76B15}
\keywords{Water waves, Solitary waves, Minimisation, FDKP equation, DS equation}

\maketitle

\begin{abstract}
We prove existence of small-amplitude modulated solitary waves for the full-dispersion Kadomtsev--Petviashvilii (FDKP) equation with weak surface tension. The resulting waves are small-order perturbations of scaled, translated and frequency-shifted solutions of a Davey--Stewartson (DS) type equation. The construction is variational and relies upon a series of reductive steps which transform the FDKP functional to a perturbed scaling of the DS functional, for which least-energy ground states are found. We also establish a convergence result showing that scalings of FDKP solitary waves converge to ground states of the DS functional as the scaling parameter tends to zero.

Our method is robust and applies to nonlinear dispersive equations with the properties that (i) their dispersion relation has a global minimum (or maximum)
at a non-zero wave number, and (ii) the associated formal weakly nonlinear analysis leads to a DS equation of elliptic-elliptic focussing type. We present full details for
the FDKP equation.

\end{abstract}

\section{Introduction}\label{sec:intro}
\subsection{Background}
In this article we consider the $(2+1)$-dimensional full-dispersion Kadomtsev--Petviashvili (FDKP) equation
\begin{equation}\label{eq:fdkp}
u_t+m(\Diff)u_x+2uu_x=0,
\end{equation}
where the Fourier multiplier operator \(m(\Diff)\) is given by
\begin{equation}\label{eq:m}
m(\Diff)=(1+\beta\abs{\Diff}^2)^\frac{1}{2}\left(\frac{\tanh(\abs{\Diff})}{\abs{\Diff}}\right)^\frac{1}{2}\left(1+\frac{2\Diff_2^2}{\Diff_1^2}\right)^\frac{1}{2}
\end{equation}
with \(\Diff=-\mathrm{i}(\partial_x,\partial_y)\). This equation was introduced by Lannes \cite[chapter 8]{Lannes} (see also Lannes \& Saut \cite{LannesSaut14}, and Pilod \emph{et al.} \cite{PilodSautSelbergTesfahun21} for a discussion of the initial-value problem)
as a model equation for weakly transversal, small-amplitude, three-dimensional water waves
which preserves the dispersion relation
\begin{equation}\label{eq:c}
c(\omega)=(1+\beta \omega^2)^\frac{1}{2}\left(\frac{\tanh(\omega)}{\omega}\right)^\frac{1}{2}
\end{equation}
for linear sinusoidal water waves with speed $c$ and wave number $\omega$; here the Bond
number $\beta$ is a dimensionless parameter measuring the relative strength of surface
tension. The FDKP equation is an alternative to the standard KP equation
\[
(u_t - 2 u u_x +\tfrac{1}{2}(\beta-\tfrac{1}{3})u_{xxx})_x - u_{yy} = 0,
\]
which is derived from \eqref{eq:fdkp} (or directly from the water-wave equations)
by  making an additional long-wave approximation.
Nontrivial solutions to the above equations are
termed \emph{steady waves} if they depend upon $x$ and $t$ only through the
combination $x-ct$, and \emph{solitary waves} are steady waves which are evanescent
in all spatial directions.

\begin{figure}
\centering
\includegraphics[scale=0.5]{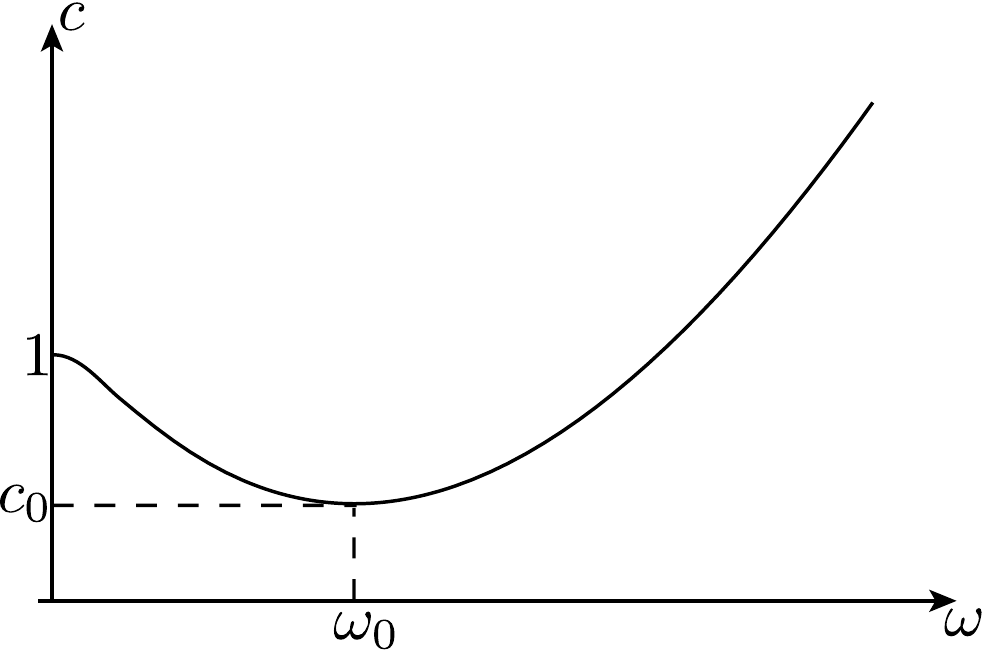}
\qquad\includegraphics[scale=0.5]{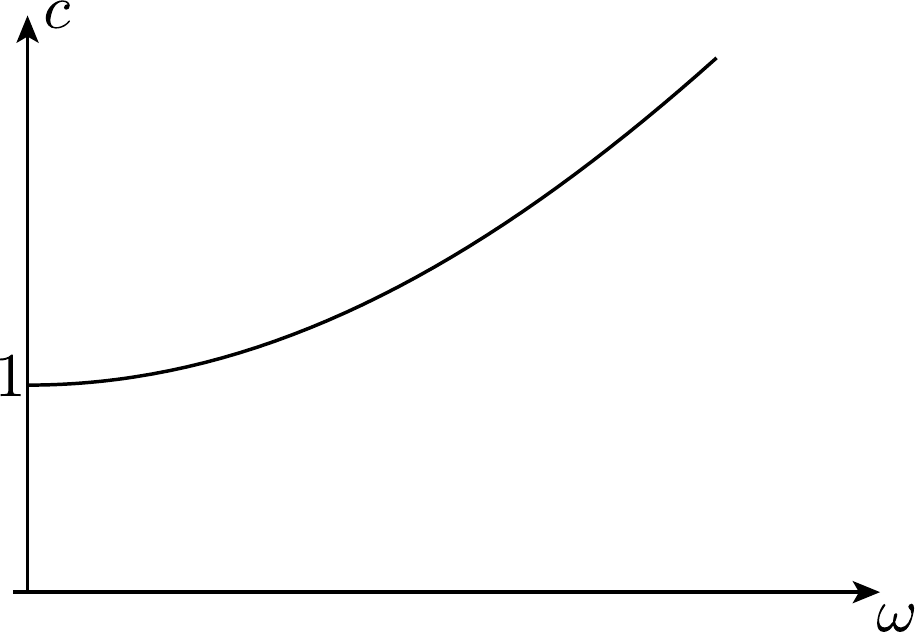}
\caption{The dispersion relation \eqref{eq:c} for 
$0<\beta<\frac{1}{3}$ (left) and $\beta>\frac{1}{3}$ (right)} \label{fig:disprel}
\end{figure}

For $0<\beta<\frac{1}{3}$ (`weak surface tension') the function $c=c(\omega)$ has a positive global minimum
$c_0 = c(\omega_0) >0$ for some $\omega_0>0$ (see Figure \ref{fig:disprel}, right), and solitary-wave
solutions to \eqref{eq:fdkp}, that is, solutions of the form
\[
u(x,y,t)=u(x-ct,y)
\]
which satisfy
\begin{equation}\label{t-fdkp}
c u+m(\Diff)u+u^2=0,
\end{equation}
can be obtained formally by a modulational ansatz. Writing $c=c_0(1-\varepsilon^2)$ and
formally expanding \(u\) as a series 
\[
u(x,y)= u_1(x, y) + u_2(x, y) + u_3(x, y) + \cdots,
\]
where
\begin{align*}
u_1(x,y)&= \re \big(\varepsilon \zeta(\varepsilon x,\varepsilon y)\exp(\mathrm{i}\omega_0 x)\big),\\
u_2(x,y)&= \re \bigg( \varepsilon^2 \sum_{j=0}^2 \zeta_{2,j}(\varepsilon x,\varepsilon y) \exp(\mathrm{i} j\omega_0 x)\bigg),\\
u_3(x,y)&= \re \bigg( \varepsilon^3 \sum_{j=0}^3 \zeta_{3,j} (\varepsilon x,\varepsilon y) \exp(\mathrm{i} j\omega_0 x)\bigg),
\end{align*}
we find from \eqref{t-fdkp} that $\zeta$ satisfies the Davey-Stewartson (DS) type equation
\begin{equation}\label{eq:ds}
-a_1\zeta_{xx}-a_2 \zeta_{yy}+a_3 \zeta-\tfrac{1}{8 \, n(2\omega_0,0 )} \abs{\zeta}^2 \zeta-\tfrac{1}{4} \zeta \left (\left(1+\tfrac{2 \Diff_2^2}{\Diff_1^2}\right)^\frac{1}{2}-c_0\right)^{-1} \abs{\zeta}^2  =0,
\end{equation}
where $n(k)=m(k)-c_0$ and
\[
a_1=\tfrac{1}{8}\partial_{k_1}^2 n(\omega_0,0),\qquad a_2=\tfrac{1}{8}\partial_{k_2}^2 n(\omega_0,0),\qquad a_3=\tfrac{1}{4} c_0.
\]
Equation \eqref{eq:ds} is of elliptic-elliptic, focussing type, and a similar equation is obtained
from the water-wave equations by the same method.
Variational
existence proofs for solitary-wave solutions to various classes of elliptic-elliptic, focussing DS equations
have been given by
Cipolatti \cite{Cipolatti92}, Wang, Ablowitz \& Segur \cite{WangAblowitzSegur94} and Papanicolaou \emph{et al.} \cite{PapanicolaouSulemSulemWang94}
(see Figure \ref{dsfloc} for a sketch of the corresponding function $u_1(x,y)$).
In this paper we rigorously establish the existence of
small-amplitude solitary-wave solutions to the FDKP equation which are approximated by scaled solutions of
\eqref{eq:ds}.

\begin{figure}
\centering
\includegraphics[scale=0.6]{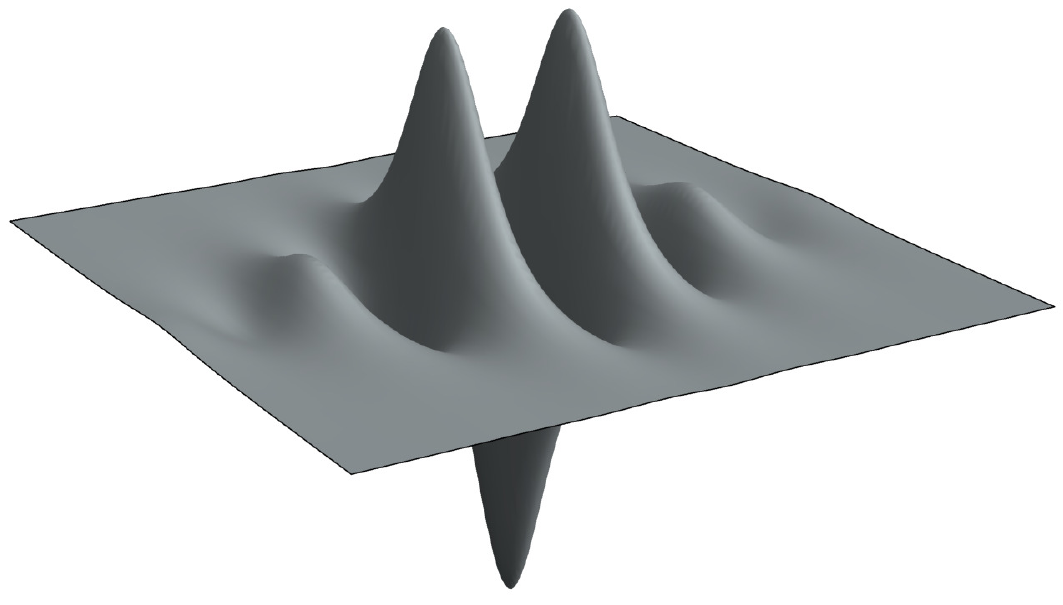}
\caption{Sketch of a modulated solitary wave of DS type} \label{dsfloc}
\end{figure}

The formal modulational ansatz above also applies to the water-wave problem
with weak surface tension in \eqref{eq:c}, and similarly leads to a
DS equation. This problem was recently treated by Buffoni, Groves \& Wahl\'{e}n \cite{BuffoniGrovesWahlen18},
who confirmed the validity of the DS approximation for solitary waves. That paper is different from, but motivated by, a series of variational technique existence results on the full water-wave problem, including Groves \& Wahl\'{e}n \cite{GrovesWahlen08} on two-dimensional gravity waves including vorticity effects, and 
Groves \& Sun \cite{GrovesSun08} and Buffoni \emph{et al.} \cite{BuffoniGrovesSunWahlen13}
on three-dimensional waves with strong surface tension. In our article we place the newer method from  reference \cite{BuffoniGrovesWahlen18} in a broader context, showing how it can be applied to equations where the modulational ansatz leads to
an elliptic-elliptic focussing DS equation and giving full details for the FDKP equation as a representative example.
As part of our analysis we also give an alternative proof of the existence of solitary-wave solutions
to DS equations of this type.

In the case $\beta>\frac{1}{3}$ (`strong surface tension') the function $c=c(\omega)$ has a global minimum $c(0)=1$
at $\omega=0$ (see Figure \ref{fig:disprel}, left) and solitary waves are obtained formally using a long-wave ansatz.
Writing $c=1-\varepsilon^2$ and formally expanding $u$ in a series
$$u(x,y)=\varepsilon^2 \zeta(\varepsilon x, \varepsilon^2 y) + \varepsilon^4 \zeta_2(\varepsilon x, \varepsilon^2 y) + \cdots,$$
one finds from \eqref{t-fdkp} that $\zeta$ satisfies the steady elliptic KP (`KP-I') equation,
which has an explicit solitary-wave solution. Ehrnstr\"{o}m \& Groves \cite{EhrnstroemGroves18} have confirmed the
validity of this approach by showing that the FDKP equation indeed has solitary-wave solutions with
speed slightly less than unity which are approximated by scaled KP-I solitary waves.

The formal derivations of `fully reduced' model equations are analogous to normal-form approaches used in local
bifurcation theory, where the structure of the linear part of the equation is used to derive the
canonical form of the nonlinear part. It is for this reason that fully reduced model equations
derived from their full-dispersion counterparts are of the same type as the corresponding model
equations derived directly from the water-wave equations. A necessary condition for local bifurcation
is that the linear part of the equation is not invertible, and the simplest case arises when its (nontrivial) kernel
is minimal. In the context of a solitary wave modulating a periodic wavetrain with wave number $\omega_0$
(with $\omega_0=0$ for long waves) these conditions state that the
wave number-wave speed map $c=c(\omega)$ satisfies $c^\prime(\omega_0)=0$ and that $c^{-1}\{c_0\}$
contains only $\pm \omega_0$, so that
$c_0=c(\omega_0)$ is a global extremum. Under these hypotheses the
ans\"{a}tze described above lead to model equations of KP or DS type which have
solitary-wave solutions if the relevant ellipticity and focussing conditions are
satisfied. It then remains to confirm \emph{a posteriori} by a rigorous
mathematical method that the original equation has a corresponding solitary-wave solution.

There are a number of existence theories for solitary-wave solutions to $(1+1)$-dimensional full-dispersion model equations for
water waves. Augmenting the Korteweg-de Vries (KdV) equation with the full-disperson symbol \eqref{eq:c} with \(\beta = 0\), one obtains the Whitham equation (see Whitham \cite[section 13.14]{Whitham}). Small-amplitude solitary-wave solutions to that equation approximated by scaled KdV solitary waves have been found by Ehrnstr\"{o}m, Groves \& Wahl\'{e}n \cite{EhrnstroemGrovesWahlen12}
and Stefanov \& Wright \cite{StefanovWright20} (see also Hildrum \cite{Hildrum20} for
low-regularity versions of the equation and Truong, Wahl\'{e}n \& Wheeler \cite{TruongWahlenWheeler21} for large-amplitude solitary waves). The gravity-capillary version of the
Whitham equation, which is obtained by setting $\Diff_2=0$ in \eqref{eq:m}, was treated by
Arnesen \cite{Arnesen16}, Maehlen \cite{Maehlen20}, Johnson \& Wright \cite{JohnsonWright20}
and Johnson, Truong \& Wheeler \cite{JohnsonTruongWheeler21}. For this equation the modulational ansatz yields
the focussing nonlinear Schr\"{o}dinger (NLS) equation for $\beta<\frac{1}{3}$, while the long-wave ansatz leads to
a KdV equation for $\beta>1/3$, and these papers indeed confirm the existence of small-amplitude solitary-wave solutions approximated by scaled NLS or KdV solitary waves. Note that the long-wave ansatz, while still leading to the KdV equation,
is insufficient in the case $\beta<\frac{1}{3}$ since $c^{-1}\{1\}$ also contains nonzero wavenumbers (see above);
solitary waves are subject to periodic disturbances at these wavenumbers and form \emph{generalised
solitary waves} which decay to periodic ripples at large distances (see  references \cite{JohnsonWright20}
and \cite{JohnsonTruongWheeler21}). Small-amplitude solitary waves for Whitham--Boussinesq equations
and full-dispersion Green--Naghdi equations have been discussed in a similar vein by respectively
Nilsson \& Wang \cite{NilssonWang19}, Dinvay \& Nilsson \cite{DinvayNilsson21}
and Duch\^ene, Nilsson and Wahl\'{e}n \cite{DucheneNilssonWahlen18}. Other $(2+1)$-dimensional
full-dispersion model equations have also been introduced, in particular systems of DS and Benny--Roskes type
(see Lannes \cite[chapter 8]{Lannes}, Obrecht \cite{Obrecht14} and Obrecht \& Saut \cite{ObrechtSaut15}); at the
time of writing it is unkown whether they admit solitary-wave solutions.

\subsection{Methodology}

Our method is variational. The steady FDKP equation \eqref{t-fdkp} (with $c=c_0(1-\varepsilon^2)$) and steady DS equation
\eqref{eq:ds} are the Euler--Lagrange equations for the variational functionals
\begin{equation}\label{eq:fdkp-functional}
\mathcal{I}_\varepsilon(u)=\frac{1}{2}\int_{\mathbb{R}^2} \left( u m(\Diff)u + c_0(\varepsilon^2 -1)u^2\right) \dx \dy+\frac{1}{3}\int_{\mathbb{R}^2}u^3 \dx \dy.
\end{equation}
and
\begin{align}
\mathcal{T}_0(\zeta)&=\int_{\mathbb{R}^2} \left( a_1\abs{\zeta_x}^2+a_2\abs{\zeta_y}^2+a_3\abs{\zeta}^2-\tfrac{1}{16 \, n(2\omega_0,0)}\abs{\zeta}^4 \right) \dx \dy \nonumber \\
&\qquad -\frac{1}{8}\int_{\mathbb{R}^2}\left(\left(1+\frac{2k_2^2}{k_1^2}\right)^\frac{1}{2}-c_0\right)^{-1}\big| \widehat{|\zeta|^2}(k)\big|^2 \dk,
\label{ds-functional}
\end{align}
which we study in the completion $X$ of the Schwarz space
$\partial_x\mathcal{S}(\mathbb{R}^2)$ with respect to the norm
$$
\norm{u}_X^2 =\int_{\mathbb{R}^2}\left(1+\frac{k_2^2}{k_1^2}+\frac{k_2^4}{k_1^2}+\abs{k}^{2s}\right)\abs{\hat{u}(k)}^2 \dk, \qquad s>\tfrac{3}{2},
$$
and the standard Sobolev space $H^1({\mathbb R}^2)$ (the choice of function spaces and related technical details
are discussed in Section \ref{sec:function spaces}). We find a nontrivial critical point of $\II_\varepsilon$ by perfoming a
rigorous local variational reduction which converts it to a perturbation of ${\mathcal T}_0$ and employing a novel method for
finding critical points of ${\mathcal T}_0$ which is robust under perturbations.

\begin{figure}
\includegraphics[scale=0.6]{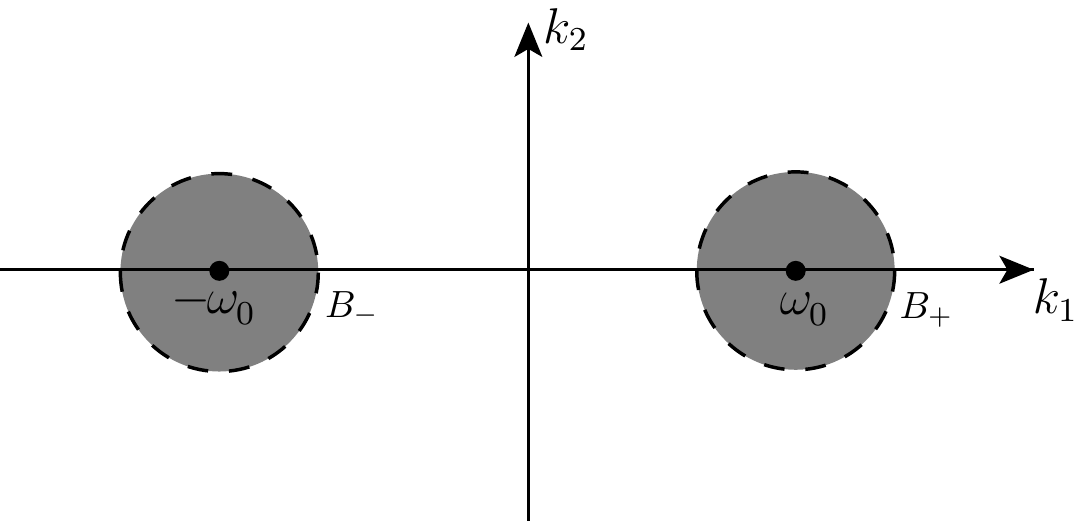}
\caption{The bi-disc $B=B_+\cup B_-$ in Fourier space}\label{fig:bidisc}
\end{figure}

The modulational ansatz suggests that the Fourier transform of a solitary wave (i.e. a nontrivial critical point of $\mathcal{I}_\varepsilon$)
is concentrated near the points
$(\omega_0,0)$ and $(-\omega_0,0)$. We therefore decompose
$$u=u_1+u_2,$$
where $u_1= \chi(\Diff)u \in X_1:= \chi(\Diff)X$, $u_2=(1-\chi(\Diff))u \in X_2:=(1-\chi(D))X$
and $\chi$ is the characteristic function of the set $B$ shown in Figure \ref{fig:bidisc}.
Observe that $u_1 \in U$ is a critical point of ${\mathcal I}_\varepsilon$, i.e.
$$\diff {\mathcal I}_\varepsilon[u](v)=0$$
for all $v$ if and only if 
$$
\diff {\mathcal I}_\varepsilon[u_1+u_2](v_1)=0, \qquad \diff {\mathcal I}_\varepsilon[u_1+u_2](v_2)=0
$$
for all $v_1 \in X_1$ and $v_2 \in X_2$.
For sufficiently small values of $\varepsilon$ the second of these equations can be solved for $u_2$ as a function of $u_1$
and we thus obtain the reduced functional
\begin{equation}
\tilde{\mathcal I}_\varepsilon(u_1):={\mathcal I}_\varepsilon(u_1+u_2(u_1)). \label{red func}
\end{equation}

In the derivation of \eqref{red func} we take $u_1 \in U_1$, where
$$U_1=\{u_1\in \overline{B}_\Lambda(0) \colon \norm{u_1}_{\dot{H}_{\omega_0}^1} \leq \varepsilon\Lambda\},$$
\(\overline{B}_\Lambda(0)\) is the closed ball of radius \(\Lambda\) in \(X_1\) and
$$
\norm{u_1}_{\dot{H}_{\omega_0}^1}=\left(\int_{\mathbb{R}^2}((\abs{k_1}-\omega_0)^2+k_2^2)\abs{\hat{u}_1}^2\ \mathrm{d}k\right)^\frac{1}{2}
$$
is the frequency-shifted analogue of the homogeneous \(\dot H^1(\R^2)\)-norm. This construction, which deviates
from the approach of Buffoni, Groves \& Wahl\'{e}n \cite{BuffoniGrovesWahlen18} necessitating a scaled $H^1(\R^2)$ norm for $u_1$, has the advantage of indicating clearly
that $\norm{u_1}_{\dot{H}_{\omega_0}^1}$ is the `small quantity' in the construction. It also allows the use of standard Gagliardo-Nirenberg inequalities, in particular
to derive the estimate
\[
\norm{u_1}_{\infty} \lesssim \norm{u_1}_{L^2}^\theta\norm{u_1}_{\dot{H}_{\omega_0}^1}^{1-\theta}, \qquad \theta \in (0,1),
\]
(see Lemma~\ref{lemma:interpolation}).
We find that
$$
\norm{u_2(u_1)}_{X} \lesssim  \varepsilon^{1-\theta} \norm{u_1}_{L^2}^{1+ \theta} \norm{u_1}_{\dot{H}_{\omega_0}^1}^{1-\theta}
$$
(with corresponding estimates for the derivatives of $u_2$). Applying the DS scaling
$$u_1(x,y)= \re \varepsilon \big(\zeta(\varepsilon x,\varepsilon y)\exp(\mathrm{i}\omega_0 x)\big)$$
and noting that
$\norm{u_1(\zeta)}_{L^2}=\norm{\zeta}_{L^2}$,
$\norm{u_1(\zeta)}_{\dot{H}_{\omega_0}^1}=\varepsilon\norm{\zeta}_{\dot{H}^1}$,
we ultimately find that
$$\varepsilon^{-2}\tilde{\mathcal I}_\varepsilon(u_1) = {\mathcal T}_\varepsilon(\zeta),$$
where
$${\mathcal T}_\varepsilon(\zeta) := {\mathcal T}_0(\zeta) + \bigO(\varepsilon^\frac{1}{2}|\zeta|_{H^1}^2)$$
(with corresponding estimates for the derivatives of the remainder term); each
critical point $\zeta$ of ${\mathcal T}_\varepsilon$ with $\varepsilon>0$ corresponds to a critical point
$u_1$ of $\tilde{\mathcal I}_\varepsilon$ which in turn defines a critical point $u_1+u_2(u_1)$ of
$\mathcal{I}_\varepsilon$.

We study ${\mathcal T}_\varepsilon$
as a functional on the set
\[
B_M(0)=\{\zeta\in H_\varepsilon^1(\mathbb{R}^2)\colon \norm{\zeta}_{H^1}\leq M\},
\]
where \(H_\varepsilon^1(\R^2)=\chi_\varepsilon(\Diff)H^1(\R^2)\), $\chi_\varepsilon$
is the characteristic function of the disc $B_{\delta/\varepsilon}(0,0)$ which contains the support of $\hat{\zeta}$
and $M$ is large enough that $B_M(0)$ contains nontrivial critical points of ${\mathcal T}_\varepsilon$;
the constant $\Lambda$ defining $U_1$ is chosen proportionally to $M$.
Note that our solutions to the FDKP equation have small amplitude but finite energy
because
$$|u|_{L^2} = |u_1|_{L^2} + |u_2|_{L^2} = 2|\zeta|_{L^2} + \bigO(\varepsilon^{1-\theta})$$
and
$$|u|_\infty \lesssim |u_1|_\infty + |u_2|_\infty \lesssim \bigO(\varepsilon^{1-\theta})$$
where we have estimated
$$|u_2(u_1)|_{L^2} \leq  |u_2(u_1)|_X$$
and
$$|u_2(u_1)|_\infty \lesssim |u_2(u_1)|_{H^s(\R^2)} \lesssim |u_2(u_1)|_X.$$

In Section \ref{sec:existence} we present a short proof of the existence
of a critical point $\zeta_\infty$ of ${\mathcal T}_\varepsilon$ by restricting it to its natural constraint set ${\mathcal N}_\varepsilon$,
noting that the critical points of $\mathcal{T}_\varepsilon$ coincide with those of $\mathcal{T}_\varepsilon|_{\NN_\varepsilon}$.
We strengthen the result in Section~\ref{sec:ground states} by showing that $\zeta_\infty$ is a ground state, i.e. a minimiser of $\mathcal{T}_\varepsilon|_{\NN_\varepsilon}$.
The theory also applies to the
case $\varepsilon = 0$ (and thus gives an alternative variational theory for solitary-wave solutions to the DS equation \eqref{eq:ds}). We exploit this fact
to show that the critical points of ${\mathcal T}_\varepsilon$ converge to critical points of ${\mathcal T}_0$
as $\varepsilon \to 0$.

We conclude this section with a summary of our main results.

\begin{theorem} \label{main theorem}
There exists \(\varepsilon_\star > 0\) with the following properties.

\begin{itemize}
\item[(i)] {\bf (The DS case.)} For each \(\varepsilon \in [0,\varepsilon_\star)\) there is a minimising sequence  \(\{\zeta_n\} \subset H^1(\R^2)\) for  $\mathcal{T}_\varepsilon$
over its natural constraint set. This sequence satisfies $\lim_{n \rightarrow \infty} \diff {\mathcal T}_\varepsilon[\zeta_n]=0$ and converges weakly in $H^1(\R^2)$ and strongly in \(L^\infty(\R^2)\) to a ground state $\zeta_\infty$.\\[-6pt]

\item[(ii)] {\bf (The FDKP case.)} Suppose that $\varepsilon \in (0,\varepsilon_\star)$. There is a mapping $\zeta \mapsto u(\zeta)$ such that $\{u(\zeta_n)\}$
converges weakly in $X$  and strongly in \(L^\infty(\R^2)\) to a nontrivial critical point $u_\infty=u(\zeta_\infty)$
of ${\mathcal I}_\varepsilon$.\\[-6pt]

\item[(iii)] {\bf (The limiting case.)} Let $\{\varepsilon_n\} \subset (0,\varepsilon_\star)$ be a sequence with $\lim_{n \to \infty} \varepsilon_n=0$ and
let $\zeta^{\varepsilon_n}$ be a ground state of ${\mathcal T}_{\varepsilon_n}$.
There exists a ground state $\zeta^\star$ of ${\mathcal T}_0$ such that
$\{\zeta^{\varepsilon_n}\}$ converges (up to subsequences and translations) in $H^1(\R^2)$ to $\zeta^\star$.
\end{itemize}
\end{theorem}

\section{Function spaces}\label{sec:function spaces}
Let  \(\Schwartz(\R^2)\) be the Schwarz space of smooth, rapidly decaying functions, and let \(\partial_x \Schwartz(\R^2) = \{ \partial_x \varphi \colon \varphi \in \Schwartz\}\) be the space of \(x\)-derivatives of such functions. The energy space \(Y\) for the FDKP functional \(\mathcal{I}_\varepsilon\) is the completion of $\partial_x\mathcal{S}(\mathbb{R}^2)$ with respect to the norm 
\begin{equation*}
\norm{u}_Y^2=\int_{\mathbb{R}^2}\left(1+\frac{\abs{k_2}}{\abs{k_1}}+\frac{\abs{k}^\frac{3}{2}}{\abs{k_1}}\right)\abs{\hat{u}(k)}^2 \dk,
\end{equation*}
while the energy space for the DS functional \(\mathcal{T}_0\) is the standard Sobolev space \(H^1(\mathbb{R}^2)\).
We shall need a smoother subspace \(X\) of \(Y\) with the property that \(m(\Diff)X \hookrightarrow L^2(\R^2)\) and consider
the image of \(X\) under \(m(\Diff)\) as a separate space. To that aim, we introduce
the completions $X$ of $\partial_x\mathcal{S}(\mathbb{R}^2)$ and $Z \cong m(\Diff) X$ of $\mathcal{S}(\mathbb{R}^2)$ with respect to the norms 
\begin{align}
\norm{u}_X^2&=\int_{\mathbb{R}^2}\left(1+\frac{k_2^2}{k_1^2}+\frac{k_2^4}{k_1^2}+\abs{k}^{2s}\right)\abs{\hat{u}(k)}^2 \dk,\label{normx}\\
\norm{u}_Z^2&=\int_{\mathbb{R}^2}(1+\abs{k}+k_1^2\abs{k}^{2s-3})\abs{\hat{u}(k)}^2 \dk,\label{normz}
\end{align}
where \(s > \frac{3}{2}\) can be chosen arbitrarily. All function spaces in this article should be considered complex-valued. (While the FDKP solution \(u\) is real, the corresponding DS solution \(\zeta\) is in general complex-valued, as are some of the Fourier transforms appearing throughout the paper.)


\begin{lemma}\label{lemma:embeddings}
One has the continuous embeddings 
\begin{align*}
&X\hookrightarrow Y\hookrightarrow L^2(\mathbb{R}^2)\\
&X\hookrightarrow H^s(\mathbb{R}^2) \hookrightarrow H^{s-\frac{1}{2}}(\mathbb{R}^2)\hookrightarrow Z\hookrightarrow L^2(\mathbb{R}^2),\\
\intertext{and}
&\qquad \qquad \:\, m(\Diff)X \hookrightarrow Z,\\ 
&\qquad \qquad\:\,  X \cdot C^{m}(\R^2) \hookrightarrow Z,
\end{align*}
for any integer $m$ with \(m > s - \frac{1}{2}\). 
\end{lemma}
\begin{proof}
The bilinear estimate is obtained by the calculation
\begin{align*}
\norm{uv}_{Z} \lesssim \norm{uv}_{H^{s-\frac{1}{2}}} \lesssim \norm{u}_{H^{s-\frac{1}{2}}} \norm{v}_{C^{m}} \lesssim \norm{u}_{H^{s}} \norm{v}_{C^{m}} \lesssim \norm{u}_X\norm{v}_{C^{m}} ,
\end{align*}
while the other results were proved by Ehrnstr\"{o}m and Groves \cite{EhrnstroemGroves18} (who
use the same definitions of $X$ and $Y$).
\end{proof}

\begin{remark} \label{rem:XZ algebra}
The calculation
$$\norm{u_1 \cdots u_n}_Z \lesssim \norm{u_1 \cdots u_n}_{H^{s-\frac{1}{2}}}
\lesssim \norm{u_1}_{H^{s-\frac{1}{2}}} \cdots \norm{u_n}_{H^{s-\frac{1}{2}}}
\lesssim \norm{u_1}_X\cdots\norm{u_n}_X$$
yields the useful estimate
$$\norm{u_1 \cdots u_n}_Z \lesssim \norm{u_1}_X\cdots\norm{u_n}_X$$
for $u_1$, \ldots, $u_n \in X$.
\end{remark}

\begin{corollary} \label{EL mapping}
The formula $u \mapsto -\varepsilon u +n(\Diff)u+ u^2$ defines a smooth and weakly continuous mapping $X \rightarrow Z$.
\end{corollary}

In accordance with the principle that there should be solitary-wave solutions of the FDKP equation
whose Fourier transform is concentrated near the frequencies \(\pm \omega_0\), let 
\[
B_\pm = B_\delta(\pm \omega_0, 0) 
\]
be discs of fixed radius \(\delta \in (0,\tfrac{\omega_0}{3} )\) centered at \((k_1,k_2) = (\pm \omega_0,0)\),
denote their characteristic functions by $\chi_\pm$, and set
\begin{equation*}
B = B_+ \cup B_-
\end{equation*}
and $\chi=\chi_++\chi_-$. The orthogonal decomposition 
\begin{equation*}
u_1=\chi(\Diff)u,\qquad u_2=(1-\chi(\Diff))u,
\end{equation*}
for functions \(u \in L^2(\R^2)\) induces the corresponding decomposition $X = X_1\oplus X_2$, where
\begin{equation*}
X_1=\chi(\Diff)X,\quad X_2=(1-\chi(\Diff))X,
\end{equation*}
and similarly for  \(Y\) and $Z$. Note that the radius of \(B_\pm\) is chosen such that the product \(u_1^2\)  belongs to \(X_2\) in spite of \(u_1 \in X_1\). Our strategy consists of reducing the main action of the FDKP equation to the low-frequency space \(X_1\), thereby retrieving the DS equation (after a frequency shift).  We now prove that the operator \(n(\Diff)=m(\Diff)-c_0\) from \eqref{eq:ds} is an isomorphism on \(X_2\), which in turn enables us to formulate the problem entirely in the low-frequency space \(X_1\). Note that \(m(\Diff)\) itself is not an isomorphism on \(X_2\), and \(n(\Diff)\) is not an isomorphism on all of \(X\).  

\begin{lemma}\label{lemma:n-isomorphism}
The mapping $n(\Diff)$ is an isomorphism $X_2\rightarrow Z_2$.
\end{lemma}
\begin{proof}
According to Lemma~\ref{lemma:embeddings},  \(n(\mathrm{D})\) is bounded \(X_2\rightarrow Z_2\), and we now show that it
has a bounded inverse. Let $\Sigma$ be the (broken) annulus
\[
\Sigma =\{k\not\in B \colon \abs{\abs{k}-\omega_0}\leq \tfrac{\delta}{2}\},
\]
intersecting the bi-disc \(B\) (Figure \ref{fig:annulus}).

\begin{figure}[h]
\centering
\includegraphics[scale=0.67]{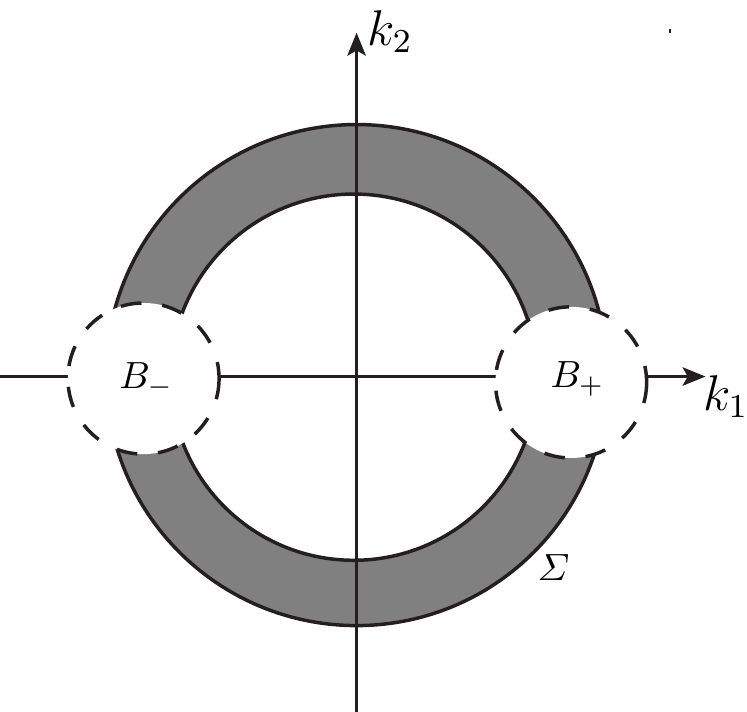}
\caption{The set $\Sigma$ in Fourier space} \label{fig:annulus}
\end{figure}

Using the linear wave speed
\[
c(\omega)=(1+\beta \omega^2)^\frac{1}{2}\left(\frac{\tanh(\omega)}{\omega}\right)^\frac{1}{2}, \qquad \omega \geq 0,
\]
we expand \(n(k)=m(k)-c_0\) as
\begin{align*}
n(k)&=c(\abs{k})\left(1+\frac{2k_2^2}{k_1^2}\right)^\frac{1}{2}-c_0\\
&=(c(\abs{k})-c_0)\left(1+\frac{2k_2^2}{k_1^2}\right)^\frac{1}{2}+c_0\left(\left(1+\frac{2k_2^2}{k_1^2}\right)^\frac{1}{2}-1\right).
\end{align*}
Because \(c_0 = c(\omega_0)\) is the strict and unique minimum of \(c\), and \(c(\omega) \eqsim \sqrt{\omega}\) for large values of \(\omega\), we have that
\begin{equation*}\label{eq:n bounded below}
n(k) \gtrsim \left(1 + |k|^{\frac{1}{2}} \right) \left(1+\frac{2k_2^2}{k_1^2}\right)^\frac{1}{2} 
\gtrsim 1+ |k|^\frac{1}{2} \left(1+\frac{k_2^2}{k_1^2}\right)^\frac{1}{2} = 1+\frac{|k|^{\frac{3}{2}}}{|k_1|}
\end{equation*}
for \(\abs{\abs{k}-\omega_0} > \tfrac{\delta}{2}\).
If on the other hand $k\in \Sigma$, then $\abs{k_2}$ is bounded away from zero and \(|k|\) is bounded from above,
so that again
\[
n(k) \gtrsim \left(1+\frac{2k_2^2}{k_1^2}\right)^\frac{1}{2}-1 \gtrsim  1 + \frac{|k_2|}{|k_1|} \gtrsim  1 + \frac{|k|^{\frac{3}{2}}}{|k_1|}.
\]
A straightforward calculation now shows that
\begin{equation*}
\left(1+\frac{k_2^2}{k_1^2}+\frac{k_2^4}{k_1^2}+\abs{k}^{2s}\right)n(k)^{-2}\lesssim 1+\abs{k}+k_1^2\abs{k}^{2s-3},
\end{equation*}
for $k$ in the complement of \(B\).
\end{proof}

Recall that \(H^1(\R^2)\) is the energy space for the DS equation \eqref{eq:ds}. One could proceed by introducing
the scaled version
$$\norm{v}_\varepsilon^2= \int_{\mathbb{R}^2}(1+\varepsilon^{-2}((\abs{k_1}-\omega_0)^2+k_2^2))\abs{\hat{v}}^2 \dk$$
of the norm for this space, which is comensurate with the DS ansatz since
$\norm{v}_\varepsilon^2 = \tfrac{1}{2}\norm{\zeta}_{H^1}^2$ for
$$v(x,z)= \tfrac{1}{2} \re \zeta(\varepsilon x,\varepsilon z) \exp(\I \omega_0 x);$$
the estimate
\begin{equation}
\norm{u_1}_{C^m}^2 \lesssim \varepsilon^2  \ln (1+\varepsilon^{-2}) \norm{u_1}_\varepsilon
\label{eq:BGW estimate}
\end{equation}
for $u_1 \in X_1$ (see below) enables
the use of fixed-point arguments to solve locally for \(u_2 \in X_2\) in terms of \(u_1 \in X_1\)
when the latter is equipped with $\norm{\cdot}_\varepsilon$. This approach is used by
Buffoni, Groves and Wahl\'{e}n \cite{BuffoniGrovesWahlen18} in their study of the water-wave problem.
In this paper we do not use the scaled norm, taking instead \(L^2(\R^2)\) as our base space, and the analogue of the homogeneous \(\dot H^1(\R^2)\)-norm, namely
\begin{equation*}
\norm{u_1}_{\dot{H}_{\omega_0}^1}=\left(\int_{\mathbb{R}^2}((\abs{k_1}-\omega_0)^2+k_2^2)\abs{\hat{u}_1}^2\ \mathrm{d}k\right)^\frac{1}{2}
\end{equation*}
as the `small quantity'.
The following lemma presents a Gagliardo--Nirenberg interpolation inequality which we use in place of
\eqref{eq:BGW estimate}.

\begin{lemma}\label{lemma:interpolation}
Fix \(\theta\in(0,1)\) and \(m \in \N_0\). The estimate
\[
\norm{u_1}_{C^{m}} \lesssim \norm{\hat{u}_1}_{L^1}
\lesssim
 \norm{u_1}_{L^2}^\theta\norm{u_1}_{\dot{H}_{\omega_0}^1}^{1-\theta}
\]
holds for all $u_1\in X_1$.
\end{lemma}

\begin{proof}
Introduce the scaled $H^1(\R^2)$-norm
\begin{equation*}
\norm{u_1}_\sigma^2= \int_{\mathbb{R}^2}(1+\sigma^{-2}((\abs{k_1}-\omega)^2+k_2^2))\abs{\hat{u}_1}^2 \dk
\end{equation*}
for functions $u_1\in X_1$ and parameter values \(\sigma > 0\). Since \(\hat u_1\) is compactly supported, we find that
\begin{align}
\norm{u_1}_{C^{m}}^2 &\lesssim \norm{\hat{u}}_{L^1}^2 \nonumber \\
&\leq \norm{u_1}_\sigma^2 \int_B \frac{\dk}{1+\sigma^{-2}((\abs{k_1}-\omega)^2+k_2^2)}  \nonumber \\
&\lesssim \sigma^2 \ln\left(1+\sigma^{-2}\right) \norm{u_1}_\sigma^2 \nonumber \\
&\lesssim \sigma^{2(1-\theta)}\norm{u_1}_\sigma^2. \label{infty-est}
\end{align}
Next note that 
\[
\sigma^{1-\theta}\norm{u_1}_\sigma\eqsim \sigma^{1-\theta}\norm{u_1}_{L^2}+\sigma^{-\theta}\norm{u_1}_{\dot{H}_{\omega_0}^1},
\]
and we now adjust $\sigma$ to the specific function $u_1$ to which the inequality is applied.
The function defined by the right-hand side of this inequality attains its global minimum at
\begin{equation*}
\sigma=\frac{\theta}{1-\theta}\frac{\norm{u_1}_{\dot{H}_{\omega_0}^1}}{\norm{u_1}_{L^2}},
\end{equation*}
where it takes the value
\begin{equation*}
\left(\tfrac{\theta}{1-\theta}\right)^{1-\theta}\norm{u_1}_{L^2}^\theta\norm{u_1}_{\dot{H}_{\omega_0}^1}^{1-\theta}+\left(\tfrac{\theta}{1-\theta}\right)^{-\theta}\norm{u_1}_{L^2}^\theta\norm{u_1}_{\dot{H}_{\omega_0}^1}^{1-\theta}.
\end{equation*}
Hence, for this choice of \(\sigma\), we find from \eqref{infty-est} that
\[\label{eq:linear interpolation}
\norm{u_1}_{C^{m}}\lesssim \norm{u_1}_{L^2}^\theta\norm{u_1}_{\dot{H}_{\omega_0}^1}^{1-\theta}.
\qedhere
\]
\end{proof}

\begin{remark} \label{rem:norms the same}
Note that $\norm{u_1}_{L^2} \eqsim \norm{u_1}_{\dot{H}_{\omega_0}^1} \eqsim \norm{u_1}_Z \eqsim \norm{u_1}_X$
for all $u_1 \in X_1$ because $|k|$ is bounded above and $|k_1|$ is bounded away from zero for
$k \in B$.
\end{remark}

\section{Variational reduction}\label{sec:reduction}
Having introduced  \(n(k)=m(k)-c_0\), one can write the steady FDKP equation \eqref{t-fdkp} as
\begin{equation}\label{t-fdkp-n}
\epsilon^2u+n(\Diff)u+u^2 = 0,
\end{equation}
and project it onto $Z_1$ and $Z_2$ using the characteristic function \(\chi\) introduced in Section~\ref{sec:function spaces}, so that
\begin{align}
\epsilon^2u_1+n(\Diff)u_1+\chi(\Diff)(u_1+u_2)^2&=0, \qquad\text{ in } Z_1, \label{eq:split 1} \\
\epsilon^2u_2+n(\Diff)u_2+(1-\chi(\Diff))(u_1+u_2)^2&=0, \qquad\text{ in } Z_2. \label{eq:split 2}
\end{align}
Our strategy is to first solve \eqref{eq:split 2} for $u_2 \in X_2$ as a function of $u_1 \in X_1$ using the following version of the contraction-mapping principle.

\begin{lemma} \label{lemma:fixedpoint}
Let \(\overline{V_1} \subset W_1\) be the closure of an open, convex, bounded neighbourhood of the origin in a Banach space $W_1$ and $r$ be a continuous function $\overline{V_1} \rightarrow [0,\infty)$. Let furthermore $F\colon \overline{V_1} \times W_2\rightarrow W_2$ be a smooth function into a Banach space \(W_2\), satisfying
\begin{equation*}
\norm{F(w_1,0)}_{W_2}\leq \tfrac{1}{2}r(w_1),\qquad \norm{\mathrm{d}_2F[w_1,w_2]}_{W_2\rightarrow W_2}\leq \tfrac{1}{3}
\end{equation*}
for all $(w_1,w_2)\in \overline{V_1} \times\overline{B}_{r(w_1)}(0)$. The fixed-point equation 
\begin{equation*}
w_2=F(w_1,w_2)
\end{equation*}
admits a smooth solution map 
\[
\overline V_1 \ni w_1 \mapsto w_2 \in \overline{B}_{r(w_1)}(0)
\]
with the properties that
\begin{align*}
\norm{\mathrm{d}w_2(w_1^1)}_{W_2} &\lesssim \norm{\mathrm{d}_1F(w_1^1)}_{W_1},\\
\norm{\mathrm{d}^2w_2(w_1^1,w_1^2)}_{W_2}&\lesssim \norm{\mathrm{d}_1^2F(w_1^1,w_1^2)}_{W_1} +\norm{\mathrm{d}_1\mathrm{d}_2F(w_1^1,\mathrm{d}w_2(w_1^2))}_{W_1}\\
&\qquad \mbox{}+\norm{\mathrm{d}_1\mathrm{d}_2F(w_1^2,\mathrm{d}w_2(w_1^1))}_{W_1}\\
&\qquad \mbox{}+\norm{\mathrm{d}_2^2F(\mathrm{d}w_2(w_1^1),\mathrm{d}w_2(w_1^2))}_{W_1},
\end{align*}
where the respective locations \(w_1\) and \((w_1,w_2(w_1))\) of the derivatives of \(w_2\) and \(F\) are implicitly assumed, and \(w_1^1,w_1^2\) are free directions in \(W_1\).
\end{lemma}

To compute the reduced equation for $u_1$ we need an explicit formula for the quadratic, $\varepsilon$-independent part of $u_2(u_1)$, which is evidently given by
\begin{equation*}
u_\mathrm{q}(u_1)=-\frac{1-\chi(\Diff)}{n(\Diff)} u_1^2=-n(\Diff)^{-1} u_1^2
\end{equation*}
because $\chi(\Diff)u_1^2=0$ in view of our choice \(\delta < \frac{\omega_0}{3}\) for the radius of the discs \(B_\pm\).
It is convenient to write $u_2=u_\mathrm{q}(u_1)+u_\mathrm{c}$ already at this stage, and
formulate \eqref{eq:split 2} as the fixed-point equation
\begin{equation}\label{eq:split G}
u_\mathrm{c}=G(u_1,u_\mathrm{c}),
\end{equation}
where
\begin{align*}
G(u_1,u_\mathrm{c})
\!\!=\!-\frac{1-\chi(\Diff)}{n(\Diff)}\!\!\left[2u_1(u_\mathrm{q}(u_1)+u_\mathrm{c})+(u_\mathrm{q}(u_1)+u_\mathrm{c})^2+\varepsilon^2(u_\mathrm{q}(u_1)+u_\mathrm{c})\right]
\end{align*}
is a smooth and weakly continuous mapping $X_1 \times X_2 \rightarrow X_2$.

\begin{remark} \label{rem:q estimate}
Using Lemmata~\ref{lemma:embeddings}--\ref{lemma:interpolation}, Remark \ref{rem:norms the same}
and the fact that\linebreak
$\mathrm{d}^2u_\mathrm{q}[u_1,u_1](u_1)=\mathrm{d}u_\mathrm{q}[u_1](u_1)=2u_\mathrm{q}(u_1)$, we find that
\[
\norm{u_\mathrm{q}(u_1)}_{X} + \norm{\mathrm{d} u_\mathrm{q}[u_1](u_1)}_{X} + \norm{\mathrm{d}^2 u_\mathrm{q}[u_1](u_1,u_1)}_{X} \lesssim |u_1|_{C^m} |u_1|_X \lesssim |u_1|_{L^2}^{1+\theta} |u_1|_{\dot H^1_\omega}^{1-\theta}.
\]
\end{remark}

We now solve \eqref{eq:split G} in \(X_2\) by applying Lemma~\ref{lemma:fixedpoint} with $W_1=X_1$, $W_2=X_2$
and $V_1$ a closed, convex subset of a fixed ball in $X_1$; Lemma~\ref{lemma:n-isomorphism} ensures that the $X$-norm of
$G(u_1,u_\mathrm{c})$ can be estimated by the $Z$-norm of the expression in square brackets on the right-hand side of the above formula (and similarly for derivatives). Here \(\Lambda > 0\) is a large, but fixed, parameter, whose value can be comfortably set later.

\begin{lemma}\label{lemma:reduction}
Let $U_1=\{u_1\in \overline{B}_\Lambda(0) \colon \norm{u_1}_{\dot{H}_{\omega_0}^1} \leq \varepsilon\Lambda\}$, where \(\overline{B}_\Lambda(0)\) is the closed ball of radius \(\Lambda\) in \(X_1\). Equation \eqref{eq:split G} defines a smooth solution
map
\[
U_1 \ni u_1 \mapsto u_\mathrm{c} \in X_2
\]
which satisfies
\begin{align*}
\norm{u_\mathrm{c}(u_1)}_{X} + \norm{\mathrm{d} u_\mathrm{c}[u_1](u_1)}_{X} + \norm{\mathrm{d}^2 u_\mathrm{c}[u_1](u_1,u_1)}_{X} & \lesssim  \varepsilon^{1-\theta} \norm{u_1}_{L^2}^{1+ \theta} \norm{u_1}_{\dot{H}_{\omega_0}^1}^{1-\theta}.
\end{align*}
\end{lemma}



\begin{proof}
First note that
\begin{align*}
\mathrm{d}_1G[u_1,u_\mathrm{c}](u_1)&=-\frac{1-\chi(\Diff)}{n(\Diff)}
\big[2(u_\mathrm{q}(u_1)+u_\mathrm{c})u_1 \notag\\
& \hspace{1.1in}\mbox{}-2(2u_1+2u_\mathrm{q}(u_1)+2u_\mathrm{c}+\varepsilon^2)u_\mathrm{q}(u_1)\big],\notag\\
\mathrm{d}_2 G[u_1,u_\mathrm{c}](u_\mathrm{c}^1)&= -\frac{1-\chi(\Diff)}{n(\Diff)} \big[(2 u_1 + 2 u_\mathrm{q}(u_1) + 2 u_\mathrm{c} +\varepsilon^2) u_\mathrm{c}^1\big],
\end{align*}
and
\begin{align*}
\mathrm{d}_1^2 G[u_1,u_\mathrm{c}](u_1,u_1)&= -2\frac{1-\chi(\Diff)}{n(\Diff)} \big[(\varepsilon^2 + 6 u_1 +6u_\mathrm{q}(u_1)+ 2u_\mathrm{c})  u_\mathrm{q}(u_1) \big],\\
\mathrm{d}_1\mathrm{d}_2G[u_1,u_\mathrm{c}](u_1,u_\mathrm{c}^1)&=
-2\frac{1-\chi(\Diff)}{n(\Diff)} [(u_1+2 u_\mathrm{q}(u_1))u_\mathrm{c}^1],\\
\mathrm{d}_2^2G[u_1,u_\mathrm{c}](u_\mathrm{c}^1,u_\mathrm{c}^2)&=-2\frac{1-\chi(\Diff)}{n(\Diff)}[u_\mathrm{c}^1 u_\mathrm{c}^2],
\end{align*}
since $\mathrm{d}u_\mathrm{q}[u_1](u_1)=2u_\mathrm{q}(u_1)$, where $u_\mathrm{c}^1$, $u_\mathrm{c}^2 \in X_2$ are free directions.

Using Lemmata~\ref{lemma:embeddings}--\ref{lemma:interpolation} and Remarks \ref{rem:norms the same}
and \ref{rem:q estimate}, we find that
\begin{align*}
\norm{G(u_1,0)}_X &
\lesssim \norm{u_1 u_\mathrm{q}(u_1)}_Z+\norm{u_\mathrm{q}(u_1)^2}_Z+\varepsilon^2\norm{u_\mathrm{q}(u_1)}_Z \\
&\lesssim \norm{u_1}_{C^m} \norm{u_\mathrm{q}(u_1)}_X + \norm{u_\mathrm{q}(u_1)}_X^2 +\varepsilon^2 \norm{u_\mathrm{q}(u_1)}_X \\
&\lesssim \norm{u_1}_{C^m} \norm{u_1}_{L^2} (\norm{u_1}_{C^m} +\norm{u_1}_{C^m} \norm{u_1}_{L^2} + \varepsilon^2) \\
&\lesssim \norm{u_1}_{L^2}^{1+\theta} \norm{u_1}_{\dot{H}_{\omega_0}^1}^{1-\theta} \left(\norm{u_1}_{L^2}^\theta \norm{u_1}_{\dot{H}_{\omega_0}^1}^{1-\theta}+\varepsilon^2 \right),
\end{align*}
where $m$ is a fixed integer with $m>s-\frac{1}{2}$. It follows that
\begin{equation}\label{eq:G(u_1,0)}
\norm{G(u_1,0)}_X \lesssim  \varepsilon^{1-\theta} \norm{u_1}_{L^2}^{1+\theta} \norm{u_1}_{\dot{H}_{\omega_0}^1}^{1-\theta}
\end{equation}
for $u_1 \in U_1$.
Similarly, 
\begin{align}
\norm{\mathrm{d}_2 G[u_1,u_\mathrm{c}](u_\mathrm{c}^1)}_X & \lesssim (\norm{u_1}_{C^m} +\norm{u_1}_{C^m} \norm{u_1}_{L^2} + \norm{u_\mathrm{c}}_X +\varepsilon^2)\norm{u_\mathrm{c}^1}_X \nonumber \\
&\lesssim (\norm{u_1}_{L^2}^\theta \norm{u_1}_{\dot{H}_{\omega_0}^1}^{1-\theta} +\norm{u_\mathrm{c}}_X)\norm{u_\mathrm{c}^1}_X \nonumber\\
&\lesssim (\varepsilon^{1-\theta} +\norm{u_\mathrm{c}}_X)\norm{u_\mathrm{c}^1}_X \label{eq:d2Gnorm}
\end{align}
for $u_1 \in U_1$. Let $r(u_1)$ be a sufficiently large multiple of the right-hand side of \eqref{eq:G(u_1,0)}, and consider $u_\mathrm{c} \in \overline{B}_{r(u_1)}(0) \subset X_2$ in \eqref{eq:d2Gnorm}. Lemma~\ref{lemma:fixedpoint} guarantees a unique fixed point $u_\mathrm{c}(u_1) \in \overline{B}_{r(u_1)}(0)$ of \eqref{eq:split G} for sufficiently small values of $\varepsilon$.

Proceeding in the same manner, and estimating $\norm{u_\mathrm{c}(u_1)}_X \leq r(u_1)$, we find that
\begin{align*}
\norm{\mathrm{d}_1G[u_1,u_\mathrm{c}(u_1)](u_1)}_{X}&\lesssim \norm{u_1 u_\mathrm{q}(u_1)}_Z +\norm {u_1 u_\mathrm{c}(u_1)}_{Z}+ \norm{u_\mathrm{q}(u_1)u_\mathrm{c}(u_1)}_Z\\
&\qquad\mbox{}+ \norm{u_\mathrm{q}(u_1)^2}_Z+\varepsilon^2 \norm{u_\mathrm{q}(u_1)}_Z\\
&\lesssim  \varepsilon^{1-\theta} \norm{u_1}_{L^2}^{1+ \theta} \norm{u_1}_{\dot{H}_{\omega_0}^1}^{1-\theta},
\end{align*}
and
\begin{align*}
\norm{\mathrm{d}_1^2 G[u_1,u_\mathrm{c}(u_1)](u_1,u_1)}_X &\lesssim \norm{u_1 u_\mathrm{q}(u_1)}_Z + \norm{u_\mathrm{q}(u_1)u_\mathrm{c}(u_1)}_Z\\
&\qquad\mbox{}+ \norm{u_\mathrm{q}(u_1)^2}_Z+\varepsilon^2 \norm{u_\mathrm{q}(u_1)}_Z\\
&\lesssim  \varepsilon^{1-\theta} \norm{u_1}_{L^2}^{1+\theta} \norm{u_1}_{\dot{H}_{\omega_0}^1}^{1-\theta}, \\
\norm{\mathrm{d}_1\mathrm{d}_2G[u_1,u_\mathrm{c}(u_1)](u_1,u_\mathrm{c}^1)}_X&\lesssim \norm{u_1 u_\mathrm{c}^1}_Z + \norm{u_\mathrm{q}(u_1) u_\mathrm{c}^1}_Z\\
&\lesssim \norm{u_\mathrm{c}^1}_X, \\
\norm{\mathrm{d}_2^2G[u_1,u_\mathrm{c}(u_1)](u_\mathrm{c}^1,u_\mathrm{c}^2)}_X&\lesssim \norm{u_\mathrm{c}^1}_X \norm{u_\mathrm{c}^2}_X, 
\end{align*}
from which the remaining estimates for $u_\mathrm{c}(u_1)$ follow by Lemma \ref{lemma:fixedpoint}.
\end{proof}

Substituting $u_2=u_\mathrm{q}(u_1)+u_\mathrm{c}(u_1)$ into equation \eqref{eq:split 1}, we obtain the
reduced equation
$$\varepsilon^2u_1+n(\Diff)u_1+\chi(\Diff)(u_1+u_\mathrm{q}(u_1)+u_\mathrm{c}(u_1))^2=0$$
for $u_1$, which is the Euler-Lagrange equation for the reduced functional
$\widetilde{\mathcal{I}}_\varepsilon\colon U_1 \mapsto \mathbb{R}$ defined by
\begin{align*}
\widetilde{\mathcal{I}}_\varepsilon(u_1)&:=\mathcal{I}_\varepsilon(u_1+u_\mathrm{q}(u_1)+u_\mathrm{c}(u_1))\\
&= \frac{1}{2} \int_{\mathbb{R}^2} \big( \sqrt{n(\Diff)} (u_1+u_\mathrm{q}(u_1)+u_\mathrm{c}(u_1)) \big)^2 \, \mathrm{d}x\, \mathrm{d}y\\ 
&\qquad\qquad\mbox{}+ \frac{c_0 \varepsilon^2}{2}\int_{\mathbb{R}^2}(u_1+u_\mathrm{q}(u_1)+u_\mathrm{c}(u_1))^2\, \mathrm{d}x\, \mathrm{d}y \\
&\qquad\qquad\mbox{}+\frac{1}{3}\int_{\mathbb{R}^2}(u_1+u_\mathrm{q}(u_1)+u_\mathrm{c}(u_1))^3\, \mathrm{d}x\, \mathrm{d}y,
\end{align*}
where the symmetrically weighted inner product $\langle \sqrt{n(\Diff)} u,\sqrt{n(\Diff}v\rangle$
is well defined for \(u, v \in X\) since \(n(k)\) is real and non-negative. In Lemma \ref{lemm:itilde_decomp} below
we identify the leading-order terms in $\widetilde{\mathcal{I}}_\varepsilon$; in its proof we use the following
technical result, which shows that higher-order nonlinear terms in our functional are small. Note that \(\norm{u_1}_{\dot{H}_{\omega_0}^1} \lesssim \varepsilon\), and \(\theta \in (0,1)\) can be taken arbitrarily small, so that the
right-hand side of the estimate is essentially dominated by \(\varepsilon^{r^* + s + 2t}\) with \(r^* \geq r-2\).

\begin{proposition}\label{prop:general estimate}
Let \(r, s,t \geq 0\) be integers with \(r + s + t \geq 2\), and let \(r^* = r - \max\{2 - (s + t), 0\}\). One has the estimate
\[
\left| \int u_1^r u_\mathrm{q}(u_1)^s u_\mathrm{c}(u_1)^t \dx \dy \right| \lesssim \varepsilon^{t(1-\theta)} \norm{u_1}_{\dot{H}_{\omega_0}^1}^{(r^* + s+t)(1-\theta)}  \norm{u_1}_{L^2}^{(r - r^* + s+t) + (r^* + s + t)\theta}
\]
and this result remains true when all or some of the \(u_\mathrm{q}(u_1)\) and \(u_\mathrm{c}(u_1)\) are replaced by
their derivatives \(\diff u_\mathrm{q}[u_1](u_1)\) or \(\diff^2 u_\mathrm{q}[u_1](u_1,u_1)\) and \(\diff u_\mathrm{c}[u_1](u_1)\) or \(\diff^2 u_\mathrm{c}[u_1](u_1,u_1)\), respectively.
\end{proposition}

\begin{proof}
Estimating the integral using the Cauchy-Schwarz inequality and Remark \ref{rem:XZ algebra}, one finds that
\begin{align*}
\left| \int u_1^r u_\mathrm{q}(u_1)^s u_\mathrm{c}(u_1)^t \dx \dy \right| & \lesssim |u_1|_{C^m}^{r^*} | u_1^{r-r^*} u_\mathrm{q}(u_1)^s |_{L^2}  |u_\mathrm{c}(u_1)^t|_{L^2}\\ 
&\lesssim |u_1|_{C^m}^{r^*} | u_1^{r-r^*} u_\mathrm{q}(u_1)^s |_{Z}  |u_\mathrm{c}(u_1)^t|_{X}\\
&\lesssim |u_1|_{C^m}^{r^*} | u_1 |_X^{r-r^*} |u_\mathrm{q}(u_1)|_{X}^s  |u_\mathrm{c}(u_1)|_{X}^t\\
&\lesssim  \varepsilon^{t(1-\theta)} \norm{u_1}_{\dot{H}_{\omega_0}^1}^{(r^* + s+t)(1-\theta)}  \norm{u_1}_{L^2}^{(r - r^* + s+t) + (r^* + s + t)\theta},
\end{align*}
where the last line follows by Remark \ref{rem:q estimate} and Lemmata \ref{lemma:interpolation} and \ref{lemma:reduction}.
The estimate evidently remains true if some or all of the $u_\mathrm{q}(u_1)$ and $u_\mathrm{c}(u_1)$ are replaced by one of their derivatives.
\end{proof}

\begin{lemma}\label{lemm:itilde_decomp}
The reduced functional $\widetilde{\mathcal{I}}_\varepsilon\colon U_1 \mapsto \mathbb{R}$ satisfies
\begin{equation*}
\widetilde{\mathcal{I}}_\varepsilon(u_1) = \frac{1}{2} |\sqrt{n(\Diff)} u_1|^2_{L^2} - \frac{1}{2} |\sqrt{n(\Diff)}\, u_\mathrm{q}(u_1)|^2_{L^2} +\frac{1}{2}c_0 \varepsilon^2 |u_1|_{L^2}^2  + \mathcal{R}_\varepsilon(u_1),
\end{equation*}
where
\begin{align*}
&\abs{\mathcal{R}_\varepsilon(u_1)} + \abs{\diff \mathcal{R}_\varepsilon[u_1](u_1)} + \abs{\diff^2 \mathcal{R}_\varepsilon[u_1](u_1,u_1)}\\
&\qquad \lesssim \big( \varepsilon^{2(1-\theta)} |u_1|_{\dot H_{\omega_0}^1}^{2(1-\theta)} + |u_1|_{\dot H_{\omega_0}^1}^{3(1-\theta)}\big) |u_1|^{2}_{L^2}.
\end{align*}
\end{lemma}

\begin{proof}
Using the fact that $u_1$ is orthogonal to $u_\mathrm{q}(u_1)$, $u_\mathrm{c}(u_1)$ and $u_1^2$ in $L^2(\R^2)$
and the relationship $u_1^2=-n(\Diff)u_\mathrm{q}(u_1)$, we find that
\begin{equation*}
\widetilde{\mathcal{I}}_\varepsilon(u_1)= \frac{1}{2} | \sqrt{n(\Diff)} u_1|_{L^2}^2 - \frac{1}{2} |\sqrt{n(\Diff)} u_\mathrm{q}(u_1) |_{L^2}^2  +\frac{1}{2}c_0\varepsilon^2 |u_1|_{L^2}^2 +\mathcal{R}_\varepsilon(u_1),
\end{equation*}
where
\begin{align*}
\mathcal{R}_{\varepsilon}(u_1)&= \textstyle \frac{1}{2} | \sqrt{n(\Diff)}\, u_\mathrm{c}(u_1)|_{L^2}^2 + \langle u_1, u_2(u_1)^2\rangle  \\
& \qquad\qquad\mbox{}+\tfrac{1}{3}\langle u_2(u_1),u_2(u_1)^2\rangle +\tfrac{1}{2}c_0\varepsilon^2\norm{ u_2(u_1)}_{L^2}^2
\end{align*}
and $u_2(u_1)=u_\mathrm{q}(u_1)+u_\mathrm{c}(u_1)$. Note that
\begin{align*}
| \sqrt{n(\Diff)}\, u_\mathrm{c}|_{L^2}^2 &= \langle n(\Diff)\, u_\mathrm{c},  u_\mathrm{c} \rangle \\
& \leq | n(\Diff)\, u_\mathrm{c} |_Z | u_\mathrm{c} |_X\\ 
&\lesssim | u_\mathrm{c} |_X^2 \\
& \lesssim \varepsilon^{2(1-\theta)} |u_1|_{\dot H^1_\omega}^{2(1-\theta)} |u_1|_{L^2}^{2(1+\theta)},
\end{align*}
and the derivatives of this term clearly satisfy the same estimates.
The remaining terms in $\mathcal{R}_\varepsilon(u_1)$
may all be estimated using Proposition~\ref{prop:general estimate}; we find that
\begin{align*}
\diff^k \langle u_1, u_2^2\rangle [u_1](u_1^{(k)}) &\lesssim \norm{u_1}_{\dot{H}_{\omega_0}^1}^{3(1-\theta)}  \norm{u_1}_{L^2}^{2 + 3\theta}, \\
\diff^k \langle u_2,u_2^2\rangle[u_1](u_1^{(k)}) &\lesssim  \norm{u_1}_{\dot{H}_{\omega_0}^1}^{3(1-\theta)}  \norm{u_1}_{L^2}^{3(1+\theta)},\\
\diff^k  \langle u_2,u_2 \rangle [u_1] (u_1^{(k)}) &\lesssim  \norm{u_1}_{\dot{H}_{\omega_0}^1}^{2(1-\theta)}  \norm{u_1}_{L^2}^{2(1+\theta)}, \qquad k = 0,1,2.\qedhere
\end{align*}
\end{proof}

We proceed to reduce the functional further by writing
\begin{equation*}
u_1= u_1^+ + u_1^-, 
\end{equation*}
where  \(u_\pm = \chi_\pm(\mathrm{D})u_1\) and $\chi_\pm$ are the characteristic functions of the discs $B_\pm$ (note that $ u_1^-=\overline{u_1^+}$ since $u_1$ is real).
More specifically, we now simplify the quadratic terms in the functional
\(\tilde {\mathcal I}_\varepsilon\) and approximate the term \(|\sqrt{n(\Diff)}\, u_\mathrm{q}(u_1)|^2_{L^2}\) by
two leading-order parts.

\begin{lemma}\label{lemma:tildefunctional}
For all $u_1\in U_1$, one has that
\begin{align*}
\widetilde{\mathcal{J}}_\varepsilon(u_1^+) &:= \widetilde{\mathcal{I}}_\varepsilon(u_1^+ + u_1^-) \nonumber\\
&= \big|\sqrt{n(\Diff)} u_1^+\big|^2_{L^2} + c_0 \varepsilon^2 \big|u_1^+\big|_{L^2}^2 - \frac{1}{n(2\omega_0,0)} \big| {u_1^+}^2 \big|_{L^2}^2\nonumber\\
&\qquad \mbox{}-2 \bigg| \bigg[\left(1+\frac{2 \Diff_2^2}{\Diff_1^2}\right)^\frac{1}{2}-c_0\bigg]^{-\frac{1}{2}} {u_1^+}^2 \bigg|_{L^2}^2 + \widetilde{\mathcal{R}}_\varepsilon(u_1^+),
\end{align*}
where
\begin{align*}
&\big| \mathcal{\widetilde R}_\varepsilon(u_1^+) \big| + \big|\diff \mathcal{\widetilde R}_\varepsilon[u_1^+](u_1^+) \big| + \big| \diff^2 \mathcal{R}_\varepsilon[u_1^+](u_1^+,u_1^+) \big|\\
&\qquad  \lesssim \big( \varepsilon^{2(1-\theta)} |u_1^+|_{\dot H_{\omega_0}^1}^{2(1-\theta)} + |u_1^+|_{\dot H_{\omega_0}^1}^{3(1-\theta)}\big) |u_1^+|^{2}_{L^2}.
\end{align*}
\end{lemma}

\begin{proof}
Since $\supp \hat{u}_1^\pm \subset B_\pm$, \(\overline{u_1^+} = u_1^-\) and $u_\mathrm{q}(u_1) = -n(\Diff)^{-1}(u_1^++u_1^-)^2$,
we find that
\begin{align*}
\widetilde{\mathcal{I}}_\varepsilon(u_1)
& = \big|\sqrt{n(\Diff)} u_1^+  \big|^2_{L^2}  + c_0 \varepsilon^2 \big|u_1^+\big|_{L^2}^2 \\
& \qquad\mbox{}
+\big|\sqrt{n^{-1}(\Diff)} (u_1^+)^2 \big|^2_{L^2} + 2 \big|\sqrt{n^{-1}(\Diff)} (u_1^+ u_1^-)^2 \big|^2_{L^2}+\mathcal{R}_\varepsilon(u_1).
\end{align*}

Noting that $\text{supp}(\mathcal{F}[(u_1^+)^2])\subset B_{2\delta}(2\omega_0,0)$,
we expand \(n^{-1}(k)\) around the centre of this disc, so that
\begin{align*}
\big|\sqrt{n^{-1}(\Diff)} &(u_1^+)^2 \big|^2_{L^2} \nonumber\\
& \hspace{-1cm}=    \frac{1}{n(2\omega_0,0)}\big| (u_1^+)^2 \big|^2_{L^2} 
+\underbrace{\int_{\mathbb{R}^2}\left(\frac{1}{n(k)}- \frac{1}{n(2\omega_0,0)}\right)
\abs{\mathcal{F}[(u_1^+)^2]}^2 \dk}_{\displaystyle
:=\widetilde{\mathcal{R}}_1(u_1^+)}
\end{align*}
Observe that \(n^{-1}\) is smooth with bounded derivative on \(B_{2\delta}(2\omega_0,0)\), so that, by the mean-value theorem, and monotonicity properties of the square root,
\begin{align*}
\left| \frac{1}{n(k)}- \frac{1}{n(2\omega_0,0)} \right|^\frac{1}{2} &\lesssim \abs{k-(2\omega_0,0)}^\frac{1}{2}
\lesssim \abs{k-(\omega_0,0)-s}^\frac{1}{2} + \abs{s-(\omega_0,0)}^\frac{1}{2}
\end{align*}
for $k \in B_{2\delta}(2\omega_0,0)$. Combining this fact with Young's inequality and the Cauchy-Schwarz
inequality, we find that
\begin{align*}
|\widetilde{\mathcal{R}}_1(u_1^+)|^\frac{1}{2} &\lesssim   \norm{\int \abs{k-(\omega_0,0)-s}^\frac{1}{2} |\hat{u}_1^+(k-s)|  |\hat{u}_1^+(s)| \ds}_{L^2} \nonumber \\ 
&\qquad \mbox{}+ \norm{\int \abs{s-(\omega_0,0)}^\frac{1}{2} |\hat{u}_1^+(k-s)|  |\hat{u}_1^+(s)| \ds}_{L^2} \nonumber \\
& = 2\norm{|k-(\omega_0,0)|^{\frac{1}{2}}\hat{u}_1^+}_{L^2}\norm{\hat{u}_1^+}_{L^1} \nonumber \\
& \lesssim \norm{|k-(\omega_0,0)|\hat{u}_1^+}_{L^2}^{\frac{1}{2}}\norm{\hat{u}_1^+}_{L^2}^\frac{1}{2} \norm{\hat{u}_1^+}_{L^1} \\
&\lesssim \norm{u_1^+}_{\dot{H}_{\omega_0}^1}^{\frac{3}{2}-\theta} \norm{u_1^+}_{L^2}^{\frac{1}{2}+\theta},
\end{align*}
where we have also used Lemma~\ref{lemma:interpolation}.

Similarly, since $\text{supp}(\mathcal{F}[u_1^+u_1^-])\subset B_{2\delta}(0,0)$ we write
$$n(k)=p(|k|,\tfrac{k_2}{k_1}), \qquad p(x_1,x_2)=c(x_1)(1+2x_2^2)^\frac{1}{2}-c_0$$
and approximate $n(k)$ by $p(0,\tfrac{k_2}{k_1})$, so that
\begin{align*}
\big|\sqrt{n^{-1}(\Diff)} &(u_1^+ u_1^-)^2 \big|^2_{L^2}\\ 
&\hspace{-1cm}=\int_{\mathbb{R}^2}\bigg[\left(1+\frac{2k_2^2}{k_1^2}\right)^\frac{1}{2}-c_0\bigg]^{-1}\abs{\mathcal{F}[u_1^+ u_1^-]}^2 \dk\\
&\quad +\underbrace{\int_{\mathbb{R}^2}\bigg(n(k)^{-1}-\bigg[\left(1+\frac{2k_2^2}{k_1^2}\right)^\frac{1}{2}-c_0\bigg]^{-1}\bigg)\abs{\mathcal{F}[u_1^+ u_1^-]}^2 \dk}_{\displaystyle :=\widetilde{\mathcal{R}}_2(u_1^+)},
\end{align*}
(note that $p(0,x_2)= (1+2x_2^2)^\frac{1}{2}-c_0 \neq 0$
since  \(c_0 \in (0,1)\)). Observing that
$$\frac{\diff}{\diff x_1}\left( \frac{1}{p(x_1,x_2)}\right)
=\frac{-c^\prime(x_1)}{(1+2x_2^2)^{\frac{1}{2}}(c(x_1)-c_0(1+2x_2^2)^{-\frac{1}{2}})^2}
$$
is bounded on the set $\{(x_1,x_2): |x_1| \leq 2\delta, x_2 \geq 0\}$, we find from the mean-value theorem that
$$
\bigg|n(k)^{-1}-\bigg[\left(1+\frac{2k_2^2}{k_1^2}\right)^\frac{1}{2}-c_0\bigg]^{-1}\bigg|
\lesssim |k|^\frac{1}{2} 
$$
for $k \in B_{2\delta}(0,0)$. Using the inequality
$$|k|^\frac{1}{2} \lesssim |k - (\omega_0,0) -s|^\frac{1}{2} + |s + (\omega_0,0)|^\frac{1}{2}$$
and proceeding as above thus yields
\begin{gather*}
\begin{aligned}
|\widetilde{\mathcal{R}}_2(u_1^+)|^\frac{1}{2} &\lesssim   \norm{\int \abs{k-(\omega_0,0)-s}^\frac{1}{2} |\hat{u}_1^+(k-s)|  |\hat{u}_1^-(s)| \ds}_{L^2}\\ 
&\qquad \mbox{}+ \norm{\int \abs{s+(\omega_0,0)}^\frac{1}{2} |\hat{u}_1^+(k-s)|  |\hat{u}_1^-(s)| \ds}_{L^2}\\
& \leq \norm{|k-(\omega_0,0)|^{\frac{1}{2}}\hat{u}_1^+}_{L^2}\norm{\hat{u}_1^-}_{L^1} + \norm{|k+(\omega_0,0)|^{\frac{1}{2}}\hat{u}_1^-}_{L^2}\norm{\hat{u}_1^+}_{L^1}\\
& = 2\norm{|k-(\omega_0,0)|^{\frac{1}{2}}\hat{u}_1^+}_{L^2}\norm{\hat{u}_1^+}_{L^1}\\
&\lesssim \norm{u_1^+}_{\dot{H}_{\omega_0}^1}^{\frac{3}{2}-\theta} \norm{u_1^+}_{L^2}^{\frac{1}{2}+\theta}.
\end{aligned}
\end{gather*}

Finally, note that $\diff \widetilde{\mathcal{R}}_j[u_1^+](u_1^+)$ and $\diff^2 \widetilde{\mathcal{R}}_j[u_1^+](u_1^+,u_1^+)$
satisfy the same estimates as $\widetilde{\mathcal{R}}_j(u_1^+)$ since these functionals are homogeneous in $u_1^+$.
The stated result thus follows by defining
\[\widetilde{R}_\varepsilon(u_1^+)=\widetilde{\mathcal R}_1(u_1^+)+\widetilde{\mathcal R}_2(u_1^+)
+\mathcal{R}_\varepsilon(u_1^++u_1^-)\]
and noting that we can replace $u_1$ by $u_1^+$ in the estimates for $\mathcal{R}_\varepsilon$
because the mapping
\(u_1^- \mapsto u_1^+\) defines isometric anti-isomorphisms
 $\chi^-(\Diff)L^2(\R^2) \rightarrow \chi^+(\Diff)L^2(\R^2)$ and  $\chi^-(\Diff)\dot H_{\omega_0}^1(\R^2) \rightarrow \chi^+(\Diff)\dot H_{\omega_0}^1(\R^2)$.
\end{proof}

The next step is to replace the symbol \(n\) with the elliptic operator
$$
\tilde{n}(k)=\tfrac{1}{2}\partial_{k_1}^2n(\omega_0,0)(k_1-\omega)^2+\tfrac{1}{2}\partial_{k_2}^2n(\omega_0,0)k_2^2
$$
appearing in the DS equation \eqref{eq:ds}. However we cannot merely expand the symbol
\(n\) around the centre of the disc \(B_+\) because the simple estimate
\begin{equation}\label{eq:n-tilde n}
n(k) = \tilde n(k) + \bigO(|k-(\omega_0,0)|^3), \qquad k \in B_+,
\end{equation}
leads to an insufficient remainder term. This difficulty is overcome using a change of variables.

\begin{lemma}\label{lemma:tilde map}
The formula
$$\tilde{u}_1^+ = \left(\frac{n(\Diff)}{\tilde n(\Diff)}\right)^{\frac{1}{2}} u_1^+$$
defines automorphisms on \(\chi^+(\mathrm{D})L^2(\mathbb{R}^2)\)
and \(\chi^+(\Diff)\dot H_{\omega_0}^1(\R^2)\).
\end{lemma}

\begin{proof}
It follows from \eqref{eq:n-tilde n} that
\[
\left| \frac{n(k)}{\tilde n(k)} - 1\right|,\ \left| \frac{\tilde{n}(k)}{n(k)} - 1\right| \lesssim |k-(\omega_0,0)| \leq \delta, \qquad k \in B_+.\qedhere
\]
\end{proof}

The next result shows that both the form of \(\widetilde{\mathcal{J}}_\varepsilon(u_1^+)\) and the estimates for its remainder term remain
unchanged when $n$ and $u_1^+$ are replaced by $\tilde{n}$ and $\tilde u_1^+$.

\begin{lemma}\label{K-lemma}
For all $u_1\in U_1$, one has that
\begin{align*}
\widetilde{\mathcal{K}}_\varepsilon( {\tilde u}_1^+) &:= \widetilde{\mathcal{J}}_\varepsilon\left({u}_1^+\right) \nonumber \\
&=
\big|\sqrt{\tilde n(\Diff)} {\tilde u}_1^+\big|^2_{L^2} + c_0 \varepsilon^2 \big| {\tilde u}_1^+ \big|_{L^2}^2 - \frac{1}{n(2\omega_0,0)} \big| ({\tilde u}_1^+)^2 \big|_{L^2}^2\nonumber\\
&\qquad -2 \bigg| \bigg[\left(1+\frac{2 \Diff_2^2}{\Diff_1^2}\right)^\frac{1}{2}-c_0\bigg]^{-\frac{1}{2}} ({\tilde u}_1^+)^2 \bigg|_{L^2}^2 + \widetilde{\mathcal E}_\varepsilon(\tilde u_1^+)
\end{align*}
where
\begin{align*}
&\big| \widetilde{\mathcal E}_\varepsilon({\tilde u}_1^+) \big| + \big|\diff \widetilde{\mathcal E}_\varepsilon[{\tilde u}_1^+]({\tilde u}_1^+) \big| + \big| \diff^2 \widetilde{\mathcal E}_\varepsilon[{\tilde u}_1^+]({\tilde u}_1^+,{\tilde u}_1^+) \big|\\
&\qquad  \lesssim \big( \varepsilon^{2(1-\theta)} + |{\tilde u}_1^+|_{\dot H_{\omega_0}^1}^{2-3\theta}\big) |{\tilde u}_1^+|_{\dot H_{\omega_0}^1} |{\tilde u}_1^+|_{L^2}.
\end{align*}
\end{lemma}

\begin{proof}
By construction
\begin{align*}
\big|\sqrt{n(\Diff)} u_1^+\big|^2_{L^2}  = \big|\sqrt{\tilde n(\Diff)} \tilde u_1^+\big|^2_{L^2},
\end{align*}
while
\[
c_0 \varepsilon^2 \big|u_1^+\big|^2_{L^2} = c_0 \varepsilon^2 \big|\tilde u_1^+\big|^2_{L^2} + c_0 \varepsilon^2 \widetilde{\mathcal E}_1(\tilde{u}_1^+),
\]
where
\begin{align*}
\widetilde{\mathcal E}_1(\tilde{u}_1^+)&=\left|\Big(\frac{\tilde n(\Diff)}{n(\Diff)}\Big)^\frac{1}{2} \tilde u_1^+\right|^2_{L^2} - \big|\tilde u_1^+\big|^2_{L^2} \\
&=
\norm{ \left| \frac{\tilde{n}(k)}{n(k)} - 1\right|^\frac{1}{2} \hat{\tilde{u}}_1^+ }_{L^2}  \\
& \lesssim 
\norm{|k-(\omega_0,0)|^{\frac{1}{2}}\hat{\tilde{u}}_1^+}_{L^2} \\
& \lesssim
\norm{\tilde{u}_1^+}_{\dot{H}_{\omega_0}^1}^{\frac{1}{2}} \norm{\tilde{u}_1^+}_{L^2}^{\frac{1}{2}}
\end{align*}
because of the estimate $\tilde{n}(k)/n(k)-1=\bigO(|k-(\omega_0,0)|)$ (see equation \eqref{eq:n-tilde n}).

Furthermore
$$
\big| (u_1^+)^2 \big|_{L^2}^2 = \big| (\tilde u_1^+)^2 \big|_{L^2}^2 + \widetilde{\mathcal E}_2(\tilde{u}_1^+),
$$
where
\begin{align*}
\widetilde{\mathcal E}_2(\tilde{u}_1^+)
&=
\norm{\bigg[ \Big(\frac{\tilde n(\Diff)}{n(\Diff)}\Big)^\frac{1}{2} \tilde u_1^+\bigg]^2}_{L^2}^2 -  \big| (\tilde u_1^+)^2 \big|_{L^2}^2 \\
&=\int_{\R^2} \bigg(\left|\Big(\frac{\tilde n(\Diff)}{n(\Diff)}\Big)^\frac{1}{2} \tilde u_1^+\right|^4 - |\tilde u_1^+|^4\bigg)\dx\dz \\
& \leq \bigg(\left|\Big(\frac{\tilde n(\Diff)}{n(\Diff)}\Big)^\frac{1}{2} \tilde u_1^+\right|_\infty^2 + |\tilde u_1^+|_\infty^2\bigg)
\bigg( \bigg|\Big(\frac{\tilde n(\Diff)}{n(\Diff)}\Big)^\frac{1}{2} \tilde u_1^+\bigg|_{L^2} + |\tilde u_1^+|_{L^2}\bigg)\\
& \qquad\quad\mbox{} \times
 \norm{ \left(\Big(\frac{\tilde n(\Diff)}{n(\Diff)}\Big)^\frac{1}{2}-1\right) \tilde u_1^+}_{L^2}\\
&\lesssim |\tilde u_1^+ |^{3-2\theta}_{\dot H^1_{\omega_0}} |\tilde u_1^+ |^{1+ 2\theta}_{L^2};
\end{align*}
between the fourth and fifth lines we have used
Lemmata \ref{lemma:interpolation} and \ref{lemma:tilde map} for the first
factor, Lemma \ref{lemma:tilde map} for the second, and the estimate
$(\tilde{n}(k)/n(k))^{\frac{1}{2}}-1 = \bigO(|k - (\omega_0,0)|)$ (see equation \eqref{eq:n-tilde n})
for the third. Similarly,
$$
 \bigg| \bigg[\left(1+\frac{2 \Diff_2^2}{\Diff_1^2}\right)^\frac{1}{2}-c_0\bigg]^{-\frac{1}{2}} (u_1^+)^2 \bigg|_{L^2}^2 = \bigg| \bigg[\left(1+\frac{2 \Diff_2^2}{\Diff_1^2}\right)^\frac{1}{2}-c_0\bigg]^{-\frac{1}{2}} ({\tilde u}_1^+)^2 \bigg|_{L^2}^2 
+\widetilde{\mathcal E}_3(\tilde{u}_1),
$$
where
\begin{align*}
\widetilde{\mathcal E}_3(\tilde{u}_1^+)
& =
 \bigg| \bigg[\left(1+\frac{2 \Diff_2^2}{\Diff_1^2}\right)^\frac{1}{2}-c_0\bigg]^{-\frac{1}{2}} \left(\!\Big(\frac{\tilde n(\Diff)}{n(\Diff)}\Big)^\frac{1}{2} \tilde u_1^+\right)^2 \bigg|_{L^2}^2  \\
& \hspace{1cm}\mbox{}
 - \bigg| \bigg[\left(1+\frac{2 \Diff_2^2}{\Diff_1^2}\right)^\frac{1}{2}-c_0\bigg]^{-\frac{1}{2}} ({\tilde u}_1^+)^2 \bigg|_{L^2}^2
\end{align*}
may be estimated in the same way as $\widetilde{\mathcal E}_2(\tilde{u}_1)$ because
$$\left(1+\frac{2 k_2^2}{k_1^2}\right)^\frac{1}{2}-c_0 \geq 1 - c_0 > 0.$$

Finally, note that $\diff \widetilde{\mathcal{E}}_j[u_1^+](u_1^+)$ and $\diff^2 \widetilde{\mathcal{E}}_j[u_1^+](u_1^+,u_1^+)$
satisfy the same estimates as $\widetilde{\mathcal{E}}_j(u_1^+)$ since these functionals are homogeneous in $u_1^+$.
The stated result thus follows by defining
\[\widetilde{\mathcal E}_\varepsilon(\tilde{u}_1^+)=\widetilde{\mathcal E}_1(\tilde{u}_1^+)+\widetilde{\mathcal E}_2(\tilde{u}_1^+)
+\widetilde{\mathcal E}_3(\tilde{u}_1^+)+\widetilde{\mathcal R}_\varepsilon
\left(\Big(\frac{\tilde n(\Diff)}{n(\Diff)}\Big)^\frac{1}{2} \tilde u_1^+\right)
\]
and noting that we can replace $u_1^+$ by $\tilde{u}_1^+$ in the estimates for $\widetilde{\mathcal R}_\varepsilon$
because of Lemma \ref{lemma:tilde map}.
\end{proof}

Finally, we apply the DS scaling by writing
\begin{equation*}
\tilde{u}_1^+(x,y)=\frac{1}{2}\varepsilon\zeta(\varepsilon x,\varepsilon y)\mathrm{e}^{\mathrm{i}\omega_0 x}.
\end{equation*}
The mapping $\tilde{u}_1^+\mapsto \zeta$ defines isomorphisms $\chi^+(\mathrm{D})L^2(\mathbb{R}^2)\rightarrow \chi_\varepsilon (\mathrm{D})L^2(\mathbb{R}^2)$ and
$\chi^+(\mathrm{D})\dot H_{\omega_0}^1(\R^2)\rightarrow \chi_\varepsilon (\mathrm{D})\dot H^1(\R^2)$,
where \(\chi_\varepsilon =\chi_{B_{\delta/\varepsilon}(0)}\) is the characteristic function of the ball of radius \(\delta/\varepsilon\),
since
\begin{equation*}
\hat{\tilde{u}}_1^+(k)=\frac{1}{2}\varepsilon^{-1}\hat{\zeta}\left(\frac{k_1-\omega_0}{\varepsilon},\frac{k_2}{\varepsilon}\right).
\end{equation*}
Defining
$$
\mathcal{T}_\varepsilon(\zeta)=\varepsilon^{-2}\widetilde{\mathcal{K}}_\varepsilon(\tilde{u}_1^+(\zeta)),
$$
we find that
\begin{equation}
\mathcal{T}_\varepsilon(\zeta)=\mathcal{T}_0(\zeta)+ {\mathcal{E}}_\varepsilon (\zeta),
\label{Def of Te}
\end{equation}
where $\mathcal{T}_0$ is the DS functional \eqref{ds-functional} and
$$
{\mathcal{E}}_\varepsilon(\zeta)=\varepsilon^{-2}\widetilde{\mathcal E}_\varepsilon(\tilde{u}_1^+(\zeta)).
$$
It follows from the calculations $\norm{\tilde{u}_1^+(\zeta)}_{L^2}=\norm{\zeta}_{L^2}$ and 
$\norm{\tilde{u}_1^+(\zeta)}_{\dot{H}_{\omega_0}^1}=\varepsilon\norm{\zeta}_{\dot{H}^1}$ and the
estimates for $\widetilde{\mathcal E}_\varepsilon$ given in Lemma \ref{K-lemma} that
\begin{align}
\abs{{\mathcal{E}}_\varepsilon (\zeta)}+\abs{\mathrm{d}{\mathcal{E}}_\varepsilon [\zeta](\zeta)}+\abs{\mathrm{d^2}{\mathcal{E}}_\varepsilon [\zeta](\zeta,\zeta)} &\lesssim   \varepsilon^{1-3\theta} \big( 1 + |\zeta|_{\dot H^1}^{2-3\theta}\big) |\zeta|_{\dot H^1} |\zeta|_{L^2} \nonumber \\
&\lesssim \varepsilon^{\frac{1}{2}} |\zeta|_{H^1}^{2}, \label{Final error estimate}
\end{align}
where we have fixed \(\theta \in (0,\frac{1}{6})\). We study ${\mathcal T}_\varepsilon$
as a functional on the set
\[
B_M(0)=\{\zeta\in H_\varepsilon^1(\mathbb{R}^2)\colon \norm{\zeta}_{H^1}\leq M\},
\]
where \(H_\varepsilon^s(\R^2)=\chi_\varepsilon(\Diff)H^s(\R^2)\) and $M$ is
large enough that $B_M(0)$ contains nontrivial critical points of ${\mathcal T}_\varepsilon$
(see Section \ref{sec:existence} below); the constant $\Lambda$ defining
$U_1$ is chosen proportionally to $M$. Altogether we have established the following result.

\begin{lemma} \label{prop:weak trace back} 
Let \(\zeta \mapsto  u = u_1 + u_2\) be the inverse of the mapping
\[
u \mapsto u_1  \mapsto u_1^+ \mapsto \tilde u_1^+ \mapsto \zeta
\]
constructed above.
\begin{itemize}
\item[(i)]
Each critical point $\zeta_\infty \in B_M(0)$ of ${\mathcal T}_\varepsilon$ defines a critical point
\begin{equation}\label{eq:def u_infty}
u_\infty = u_1(\zeta_\infty) + u_2(u_1(\zeta_\infty))
\end{equation}
of ${\mathcal I}_\varepsilon$ in $U_1$, and any critical point $u_1$ of $\mathcal I_\varepsilon$ with $u_1 \in U_1$ defines a unique
critical point $\zeta_\infty \in B_M(0)$ of ${\mathcal T}_\varepsilon$.\\[-8pt]

\item[(ii)]
Each Palais--Smale sequence $\{\zeta_n\} \subset B_M(0)$ for ${\mathcal T}_\varepsilon$ generates a Palais--Smale
sequence $\{u_n\} \subset U_1$ for ${\mathcal I}_\varepsilon$, where
\begin{equation}\label{eq:def u_n}
u_n = u_1(\zeta_n) + u_2(u_1(\zeta_n)).
\end{equation}

\item[(iii)]
Suppose that $\{\zeta_n\} \subset B_M(0)$ converges weakly in $H^1_\varepsilon(\R^2)$ to
$\zeta_\infty \in B_M(0)$. The corresponding sequence $\{u(\zeta_n)\}$ given by \eqref{eq:def u_n} converges weakly in
$X$  to \(u_\infty = u(\zeta_\infty)\) given by \eqref{eq:def u_infty}.
\end{itemize}
\end{lemma}
\begin{proof} It remains only to establish (iii).  To this end
note that $\{u_{1,n}\}$ converges weakly in $X_1$ to
$u_{1,\infty}=u_1(\zeta_\infty) \in U_1$ and $\{u_\mathrm{q}(u_{1,n})\}$ converges
weakly in $X_2$ to $u_\mathrm{q}(u_{1,\infty})$. Furthermore, $u_{\mathrm{c},n}=u_\mathrm{c}(u_{1,n})$
is the unique solution in $X_2$ of equation \eqref{eq:split G} with $u_1=u_{1,n}$, so that
\[
u_{\mathrm{c},n}=G(u_{1,n},u_{\mathrm{c},n}).\]
Observe that $\{u_{\mathrm{c},n}\}$ is bounded in $X_2$, and
suppose that (a subsequence of) $\{u_{\mathrm{c},n}\}$ converges weakly in $X_2$ to $u_{\mathrm{c},\infty}$;
it follows that
\[
u_{\mathrm{c},\infty}=G(u_{1,\infty},u_{\mathrm{c},\infty})
\]
(because $G: X_1 \times X_2 \to X_2$ is weakly continuous),
so that $u_{\mathrm{c},\infty} = u_\mathrm{c}(u_{1,\infty})$ (the fixed-point equation
$u_\mathrm{c} = G(u_{1,\infty}, u_\mathrm{c})$ has a unique solution in $X_2$).
This argument shows that any 
weakly convergent subsequence of $\{u_{\mathrm{c},n}\}$ has weak limit $u_\mathrm{c}(u_{1,\infty})$,
so that $\{u_{\mathrm{c},n}\}$ itself converges weakly to $u_\mathrm{c}(u_{1,\infty})$ in $X_2$.
Altogether we conclude that $\{u_{1,n}+u_{2,n}\}$ with $u_{2,n}=u_\mathrm{q}(u_{1,n})+u_{\mathrm{c},n}$ converges
weakly in $X$ to $u_\infty=u_{1,\infty} + u_{2,\infty}$ with $u_{2,\infty}=u_\mathrm{q}(u_{1,\infty})+u_{\mathrm{c},\infty}$.
\end{proof}

\section{Existence theory} \label{sec:existence}

According to \eqref{Def of Te} and \eqref{Final error estimate} the functional ${\mathcal T}_\varepsilon\colon B_M(0) \subset H_\varepsilon^1(\R^2) \to \R$ is a perturbation
of the `limiting' functional ${\mathcal T}_0: H_1^\varepsilon(\R^2) \to \R$ defined in \eqref{ds-functional}. More precisely
$\mathcal{E}_\varepsilon \circ\chi_\varepsilon(\Diff)$ (which coincides with
$\mathcal{E}_\varepsilon$ on $B_M(0) \subset H^1_\varepsilon(\R^2)$) converges uniformly
to zero over $B_M(0) \subset H^1(\R^2)$, and corresponding statements for its derivatives also hold.
In this section we study ${\mathcal T}_\varepsilon$ by perturbative arguments in this spirit.
We write
\begin{equation}
{\mathcal T}_0(\zeta) = {\mathcal Q}(\zeta) - {\mathcal S}(\zeta),
\label{eq:final red func}
\end{equation}
where
\begin{equation}\label{eq:Q}
{\mathcal Q}(\zeta) = \int_{\R^2}(a_1 |\zeta_x|^2 + a_2 |\zeta_y|^2 + a_3 |\zeta|^2) \dx\dy
\end{equation}
is a local quadratic term equivalent to  \(|\zeta|_{H^1}^2\), and
\[
{\mathcal S}(\zeta) =  \big| \sqrt{L(\Diff)} |\zeta|^2 \big|_{L^2}^2
\]
with
\[
L(\Diff)=\frac{1}{16 n(2\omega_0,0)}+\frac{1}{8}\left[\left(1+\frac{2 \Diff_2^2}{\Diff_1^2}\right)^\frac{1}{2}-c_0\right]^{-1}
\]
is a nonlocal quartic term equivalent to \(||\zeta|^2|_{L^2}^2\). (Note that $L(\Diff)$ is a
nonlocal, zeroth-order operator because  \(n(2\omega_0,0) > 0\) and  \(c_0 \in (0,1)\),
so that the continuous
bilinear form \(\langle L(\Diff) \cdot, \cdot \rangle\) defines an equivalent norm for \(L^2(\R^2)\).)

We seek critical points of ${\mathcal T}_\varepsilon$ by considering its \emph{natural constraint set}
\[
N_\varepsilon = \left\{ \zeta \in B_M(0) \colon \zeta \neq 0, \diff {\mathcal T}_\varepsilon[\zeta](\zeta) = 0 \right\},
\]
whose well-known geometrical interpretation and variational property are recorded in the
following result (e.g.\ see Buffoni, Groves \& Wahl\'{e}n \cite[p.\ 806]{BuffoniGrovesWahlen18}).

\begin{proposition} \label{prop:nc props}  \hspace{2cm}
\begin{itemize}
\item[(i)]
Any ray in $B _M(0)\setminus \{0\} \subset H^1_\varepsilon(\R^2)$ intersects
$N_\varepsilon$ in at most one point. The value of ${\mathcal T}_\varepsilon$ along such a ray attains a strict maximum at this point. (When $\varepsilon=0$ one may take $M=\infty$ and
in that case every ray intersects $N_0$ in precisely one point.)
\item[(ii)]
Any nontrivial critical point of ${\mathcal T}_\varepsilon$ lies
on $N_\varepsilon$, and conversely any critical point of ${\mathcal T}_\varepsilon|_{N_\varepsilon}$
is a (necessarily nontrivial) critical point of ${\mathcal T}_\varepsilon$.
\end{itemize}
\end{proposition}

In view of Proposition \ref{prop:nc props}(ii) we
proceed by seeking a \emph{ground state}, that is, a minimiser $\zeta^\star$
of ${\mathcal T}_\varepsilon$ over $N_\varepsilon$. We make frequent use of the identities
\begin{eqnarray}
{\mathcal T}_\varepsilon(\zeta)
&=&\tfrac{1}{2} {\mathcal Q}(\zeta)+\tfrac{1}{4}\diff {\mathcal T}_\varepsilon[\zeta](\zeta)
+ {\mathcal E}_\varepsilon(\zeta)-\tfrac{1}{4}\diff {\mathcal E}_\varepsilon[\zeta](\zeta), \label{eq:only Q} \\
{\mathcal T}_\varepsilon(\zeta) &=&
{\mathcal S}(\zeta)+\tfrac{1}{2}\diff {\mathcal T}_\varepsilon[\zeta](\zeta)
+ {\mathcal E}_\varepsilon(\zeta)
-\tfrac{1}{2}\diff {\mathcal E}_\varepsilon[\zeta](\zeta), \label{eq:only S}
\end{eqnarray}
which are obtained using the calculation
\begin{equation}\label{eq:dT}
\diff {\mathcal T}_\varepsilon[\zeta](\zeta)  = 2 {\mathcal Q}(\zeta)  - 4 {\mathcal S}(\zeta) + \diff {\mathcal E}_\varepsilon[\zeta](\zeta),
\end{equation}
to eliminate respectively ${\mathcal S}(\zeta)$ and ${\mathcal Q}(\zeta)$
from \eqref{eq:final red func}.
We begin with some
\emph{a priori} bounds for ${\mathcal T}_\varepsilon|_{N_\varepsilon}$.

\begin{proposition} \label{prop:lower bounds}
The estimates
\[
{\mathcal T}_\varepsilon(\zeta) \geq \tfrac{a}{4}  |\zeta|_{H^1}^2, \qquad\quad  |\zeta|_{H^1} \gtrsim 1, 
\]
where $a=\min(a_1,a_2,a_3)$, hold for all $\zeta \in N_\varepsilon$,
and ${\mathcal T}_\varepsilon(\zeta)<\frac{a}{4} (M-1)^2$, $\zeta \in N_\varepsilon$ implies
that $|\zeta|_{H^1}<M-1$.
\end{proposition}

\begin{proof}
Let $\zeta \in N_\varepsilon$. Using \eqref{eq:Q} and \eqref{eq:only Q}, we find that
$$
{\mathcal T}_\varepsilon(\zeta)
\geq \tfrac 1 4 {\mathcal Q}(\zeta)\\
\geq \tfrac a 4 |\zeta|_{H^1}^2,
$$
so that in particular ${\mathcal T}_\varepsilon(\zeta)<\frac{a}{4}(M-1)^2$ implies that $|\zeta|_{H^1} <M-1$. The lower bound on $|\zeta|_{H^1}$ follows from the estimate
\[
|\zeta|_{H^1}^2 \lesssim  {\mathcal Q}(\zeta) \lesssim {\mathcal S}(\zeta) + |\mathrm{d}{\mathcal E}_\varepsilon[\zeta](\zeta) |
\lesssim |\zeta|_{H^1}^4 + \varepsilon^{1/2}|\zeta|_{H^1}^2,
\]
in which we have used \eqref{eq:dT}.
\end{proof}

\begin{remark} \label{rem:inf is positive}
Let $\tau_\varepsilon := \inf_{N_\varepsilon} {\mathcal T}_\varepsilon$.
It follows from Proposition \ref{prop:lower bounds} that
$\liminf_{\varepsilon \to 0} \tau_\varepsilon { \gtrsim 1}$
and from \eqref{eq:only S} that
${\mathcal S}(\zeta) \geq \tau_\varepsilon - \bigO(\varepsilon^\frac{1}{2} |\zeta|_{H^1}^2)$ for all $\zeta \in N_\varepsilon$.
\end{remark}

The next result shows how to choose \(M\) and how points on $N_0$ may be approximated by points on $N_\varepsilon$.

\begin{proposition} \label{prop:approximate N0}
Choose $\zeta_0 \in H^1(\R^2)\setminus\{0\}$ and $M$ large enough that
\[
\frac{{\mathcal Q}(\zeta_0)^2}{{\mathcal S}(\zeta_0)}< a(M-1)^2
\]
There exists a unique point $\lambda_0 \zeta_0 \in B_{M-1}(0)$ on the ray through $\zeta_0$
which lies on $N_0$. Furthermore, there exists $\xi_\varepsilon \in N_\varepsilon$ such that
$\lim_{\varepsilon \to 0} |\xi_\varepsilon -\lambda_0\zeta_0|_{H^1} = 0$.
\end{proposition}

\begin{proof}
The calculation
$$\mathrm{d}{\mathcal T}_0[\lambda_0 \zeta_0](\lambda_0\zeta_0)
= 2\lambda_0^2 {\mathcal Q}(\zeta_0) - 4\lambda_0^4 {\mathcal S}(\zeta_0)
$$
shows that $\lambda_0\zeta_0 \in N_0$ with
\[
\lambda_0=\left(\frac{{\mathcal Q}(\zeta_0)}{2{\mathcal S}(\zeta_0)}\right)^{1/2}.
\]
It follows that $\lambda_0\zeta_0$ is the unique point on its ray which lies on $N_0$, 
so that 
\begin{equation}
\frac{\diff }{\diff  \lambda} {\mathcal T}_0(\lambda \zeta_0) \Big|_{\lambda=\lambda_0} = 0,
\qquad
\frac{\diff ^2}{\diff  \lambda^2} {\mathcal T}_0(\lambda \zeta_0) \Big|_{\lambda=\lambda_0} < 0.
\label{eq:zero ray}
\end{equation}
(see Proposition \ref{prop:nc props}(i))
and
\begin{equation}
{\mathcal T}_0(\lambda_0\zeta_0)=\tfrac{1}{2}{\mathcal Q}(\lambda_0\zeta_0)=\frac{{\mathcal Q}(\zeta_0)^2}{4{\mathcal S}(\zeta_0)}<\frac{a}{4} (M-1)^2, \label{eq:value of J0}
\end{equation}
so that
\[|\lambda_0 \zeta_0|_{H^1} < M-1\]
according to Proposition \ref{prop:lower bounds}.

Let $\zeta_\varepsilon = \chi_\varepsilon(\Diff)\zeta_0$, so that
$\zeta_\varepsilon\in H^1_\varepsilon(\R^2) \subset H^1(\R^2)$ with
$\lim_{\varepsilon\to 0}|\zeta_\varepsilon-\zeta_0|_{H^1}=0$, and in particular
\[|\lambda_0\zeta_\varepsilon|_{H^1} < M-1.\]
According to \eqref{eq:zero ray} we can find $\tilde{\gamma}>1$ such that $\tilde{\gamma} |\lambda_0 \zeta_\varepsilon |_{H^1} < M$
(so that $\tilde{\gamma}\lambda_0\zeta_\varepsilon \in B_M(0)$) and
\[
\frac{\diff }{\diff  \lambda} {\mathcal T}_0(\lambda \zeta_0) \Big|_{\lambda=\tilde{\gamma}^{-1}\lambda_0}>0,
\qquad
\frac{\diff }{\diff  \lambda} {\mathcal T}_0(\lambda \zeta_0) \Big|_{\lambda=\tilde{\gamma}\lambda_0}<0,
\]
and therefore
\[
\frac{\diff }{\diff  \lambda} {\mathcal T}_\varepsilon(\lambda \zeta_\varepsilon) \Big|_{\lambda=\tilde{\gamma}^{-1}\lambda_0}>0,
\qquad
\frac{\diff }{\diff  \lambda} {\mathcal T}_\varepsilon(\lambda \zeta_\varepsilon) \Big|_{\lambda=\tilde{\gamma}\lambda_0}<0,
\]
by the continuity of \(\mathcal{T}_\varepsilon\) and its derivatives. It follows that there exists $\lambda_\varepsilon\in (\widetilde \gamma^{-1}\lambda_0,\widetilde \gamma \lambda_0)$
with
\[
\frac{\diff }{\diff  \lambda} {\mathcal T}_\varepsilon(\lambda \zeta_\varepsilon) \Big|_{\lambda=\lambda_\varepsilon}=0,
\]
that is, $\xi_\varepsilon:=\lambda_\varepsilon \zeta_\varepsilon \in N_\varepsilon$, and we conclude that
this value of $\lambda_\varepsilon$ is unique (see Proposition \ref{prop:nc props}(i))
and that $\lim_{\varepsilon \to 0} \lambda_\varepsilon = \lambda_0$.
\end{proof}

\begin{corollary}\label{cor:limsup < M-1}
Any minimising sequence $\{\zeta_n\}$ of ${\mathcal T}_\varepsilon|_{N_{\varepsilon}}$ satisfies
\[
\limsup_{n\to\infty} |\zeta_n|_{H^1} <M-1.
\]
\end{corollary}

\begin{proof} Proposition~\ref{prop:approximate N0} asserts the existence of $\xi_\varepsilon \in N_\varepsilon$ with
$\lim_{\varepsilon \to 0}|\xi_\varepsilon - \lambda_0\zeta_0|_{H^1}=0.$
The continuity of $T_\varepsilon \circ \chi_\varepsilon(\Diff)$ with respect to $\varepsilon$ yields
\[
\lim_{\varepsilon\to 0}
{\mathcal T}_\varepsilon(\xi_\varepsilon)
={\mathcal T}_0(\lambda_0\zeta_0),
\]
and with \eqref{eq:value of J0} we find that
\[
{\mathcal T}_\varepsilon(\xi_\varepsilon)<\tfrac{a}{4} (M-1)^2.
\]
In view of Proposition~\ref{prop:lower bounds} this estimate shows that \(\{\zeta_n\} \subset B_{M-1}(0)\) for sufficiently large values of $n$.
\end{proof}

The next step is to show that there is a minimising sequence for ${\mathcal T}_\varepsilon|_{N_\varepsilon}$
which is also a Palais--Smale sequence.

\begin{proposition}\label{prop:minimising sequence}
There exists a minimising sequence $\{\zeta_n\} \subset  B_{M-1}(0)$
of ${\mathcal T}_\varepsilon|_{N_\varepsilon}$ such that 
\[
\lim_{n \to \infty}|\diff {\mathcal T}_\varepsilon[\zeta_n]|_{H^1 \to \R} = 0.
\]
\end{proposition}

\begin{proof}
Ekeland's variational principle for optimisation problems with regular constraints
\cite[Thm 3.1]{Ekeland74} implies the existence of
a minimising sequence $\{\zeta_n\}$ for ${\mathcal T}_\varepsilon|_{N_\varepsilon}$ and a sequence
$\{\mu_n\}$ of real numbers such that
\[
\lim_{n \to \infty} |\diff {\mathcal T}_\varepsilon[\zeta_n] - \mu_n \, \diff {\mathcal G}_\varepsilon[\zeta_n]|_{H^1 \to \R}  = 0,
\]
where ${\mathcal G}_\varepsilon = \diff {\mathcal T}_\varepsilon[\zeta_n](\zeta_n)$.
Applying this sequence of operators to $\zeta_n$, we find that $\mu_n \to 0$ as $n \to \infty$
(since $\diff {\mathcal T}_\varepsilon[\zeta_n](\zeta_n)=0$ and
\[
\diff{\mathcal G}_\varepsilon[\zeta_n](\zeta_n)
\!=\! \diff^2{\mathcal T}_\varepsilon[\zeta_n](\zeta_n,\zeta_n)
\!=\! -4{\mathcal Q}(\zeta_n) -4 \diff \mathcal{E}_\varepsilon[\zeta_n](\zeta_n) + \diff^2 \mathcal{E}_\varepsilon[\zeta_n](\zeta_n,\zeta_n)
\!\lesssim\! -1 ),
\]
whence $|\diff {\mathcal T}|_\varepsilon[\zeta_n]_{\tilde Y_\varepsilon \to \R} \to 0$ as $n \to \infty$.
\end{proof}

We proceed by using the local spaces $L^2(Q_j)$ and $H^1(Q_j)$,where
\[
Q_j = \{(x,y) \in \R^2: |x-j_1| < \tfrac{1}{2}, |y-j_2| < \tfrac{1}{2}\}
\] 
is the unit cube centered at the point $j = (j_1,j_2) \in \Z^2$,
to examine the convergence properties of general {Palais--Smale} sequences.

\begin{lemma}\label{lemma:critical points} \hspace{2cm}
\begin{itemize}
\item[(i)]
Suppose that $\{\zeta_n\} \subset B_{M-1}(0)$ satisfies
\[\lim_{n \to \infty} \diff {\mathcal T}_\varepsilon[\zeta_n]=0,
\qquad
\sup_{j \in \Z^2} |\zeta_n|_{L^2(Q_j)} \gtrsim 1.\]
There exists $\{w_n\} \subset \Z^2$ with the property that a subsequence of
$\{\zeta_n(\cdot+w_n)\}$
converges weakly in $H^1_\varepsilon(\R^2)$ to a nontrivial critical point $\zeta_\infty$ of
${\mathcal T}_\varepsilon$.\\[-6pt]

\item[(ii)]
Suppose that $\varepsilon>0$. The corresponding sequence  
\[
\qquad  u_n = u_1(\zeta_n) + u_2(u_1(\zeta_n)), \qquad  \zeta_n = \zeta_n(\cdot+w_n),
\] 
converges weakly in $X$  to a nontrivial critical point 
\[
u_\infty = u_1(\zeta_\infty) + u_2(u_1(\zeta_\infty))
\] 
of ${\mathcal I}_\varepsilon$.
\end{itemize}
\end{lemma}

\begin{proof}
We can select $\{w_n\} \subset \Z^2$ so that
\[
\liminf_{n \to \infty} |\zeta_n(\cdot+w_n)|_{L^2({Q_0})} { \gtrsim 1}.
\]
The sequence $\{\zeta_n(\cdot+w_n)\} \subset B_{M-1}(0)$
admits a subsequence which converges weakly in $H^1_\varepsilon(\R^2)$,
strongly in $L^2({Q_0})$ and pointwise almost everywhere to $\zeta_\infty \in B_M(0)$;
it follows that $|\zeta_\infty|_{L^2({Q_0})} > 0$.
We henceforth abbreviate $\{\zeta_n(\cdot+w_n)\}$ to $\{\zeta_n\}$ and extract further subsequences
as necessary.

Observe that $\{|\zeta_n|^2\}$ converges weakly in
$L^2(\R^2)$ and pointwise almost
everywhere to $|\zeta_\infty|^2$ (it is bounded in $L^2(\R^2)$ since $\{\zeta_n\}$ is bounded
in $H^1(\R^2)$ and hence in $L^4(\R^2)$). The weak convergence of \(\zeta_n\) and \(|\zeta_n|^2\) in \(L^2(\R^2)\) thus yields
\begin{align*}
& \limsup_{n \to \infty} \Big| \big\langle L(\Diff) |\zeta_n|^2, \zeta_n \overline w  \big\rangle - \big\langle L(\Diff) |\zeta_\infty|^2, \zeta_\infty \overline{w} \big\rangle \Big| \\
&\quad =  \limsup_{n \to \infty} \Big| \big\langle  |\zeta_n|^2-|\zeta_\infty|^2, L(\Diff)\zeta_\infty \overline{w} \big\rangle +   \big\langle L(\Diff) |\zeta_n|^2, (\zeta_n-\zeta_\infty) \overline w \big\rangle\Big|\\
&\quad\leq   \limsup_{n \to \infty} \Big|\big\langle L(\Diff) |\zeta_n|^2, (\zeta_n-\zeta_\infty) \overline w \big\rangle \Big| \\
& \quad \lesssim    \limsup_{n \to \infty}\big||\zeta_n|^2\big|_{L^2} \lim_{n \to \infty} \big| (\zeta_n-\zeta_\infty) \overline w \big|_{L^2} \\
& \quad= 0
\end{align*}
for each $w \in H^1(\R^2)$,
where in the final calculation we have written\linebreak
\(
|\zeta_n - \zeta_\infty|^2  = |\zeta_n|^2 - 2 \re \zeta_n \overline{\zeta_\infty} + |\zeta_\infty|^2
\).
It follows that $\diff {\mathcal S}[\zeta_n](w) \rightarrow \diff {\mathcal S}[\zeta_\infty](w)$
(because $\diff {\mathcal S}[\zeta](w)=
4 \re \langle L(\Diff)|\zeta|^2,\zeta w \rangle$) and furthermore
$\diff {\mathcal Q}[\zeta_n](w) \rightarrow
\diff {\mathcal Q}[\zeta_\infty](w)$.
In the case $\varepsilon=0$ we conclude that
$\diff {\mathcal T}_0[\zeta_n](w) \rightarrow
\diff {\mathcal T}_0[\zeta_\infty](w)$
as $n \rightarrow \infty$ for all $w \in H^1(\R^2)$, so that
$\diff {\mathcal T}_0[\zeta_\infty]=0$ by uniqueness of limits.

We cannot use
the above argument for $\varepsilon>0$ since we have not established that
$\diff {\mathcal E}_\varepsilon[\zeta_n](w) \rightarrow \diff {\mathcal E}_\varepsilon[\zeta_\infty](w)$
as $n \rightarrow \infty$, and we proceed by considering the FKDP functional
${\mathcal I}_\varepsilon$ (which has no remainder term). 
According to Lemma \ref{prop:weak trace back}(iii)
the sequence $\{u_n\}$ converges weakly in $X$ to $u_\infty$, and Lemma \ref{prop:weak trace back}(ii) shows that
\[
\lim_{n \to \infty}|\diff {\mathcal I}_\varepsilon[u_n]|_{X \to \R} =0.
\]
Since $u \mapsto \varepsilon u + n(\Diff)u + u^2$ is in particular weakly continuous
$X \to L^2(\R^2)$ (see Corollary \ref{EL mapping}), one finds that
\begin{eqnarray*}
\diff {\mathcal I}_\varepsilon[u_\infty](w) &= & \int_{\R^2} \left( \varepsilon^2 u_\infty + n(\Diff) u_\infty + u_\infty^2 \right) w \dx \dy\nonumber \\
&=&  \lim_{n \to \infty} \int_{\R^2} \left( \varepsilon^2 u_n + n(\Diff) u_n + u_n^2 \right) w \dx \dy \nonumber \\
& = & \lim_{n \to \infty}\diff {\mathcal I}_\varepsilon[u_n](w) \nonumber \\
& = & 0
\end{eqnarray*}
for any $w \in X$, whence $u_\infty$ is a critical point of ${\mathcal I}_\varepsilon$.
\end{proof}

Although we know from Proposition~\ref{prop:lower bounds} that the natural constraint set is bounded from below in \(H^1(\R^2)\), it remains to show that this bound 
implies that the minimising sequence
for ${\mathcal T}_\varepsilon$ over $N_\varepsilon$ identified in Proposition
\ref{prop:minimising sequence} satisfies the `non-vanishing' criterion in Lemma
\ref{lemma:critical points}. This task is accomplished in Proposition \ref{prop:1/2}
and Corollary \ref{cor:no vanishing} (with Remark~\ref{rem:inf is positive}).

\begin{proposition} \label{prop:1/2}
The inequality
\[
\left\langle |\rho_1|^2-|\rho_2|^2, L(\Diff)|\xi|^2  \right\rangle
\lesssim \sup_{j \in \Z} |\rho_1-\rho_2|_{L^2(Q_j)}^\frac{1}{2}(|\rho_1|_{H^1}+|\rho_2|_{H^1})^\frac{3}{2} |\xi|_{H^1}^2
\]
holds for all $\rho_1$, $\rho_2$, $\xi \in H^1(\R^2)$. 
\end{proposition}
 
\begin{proof}
First note that
\begin{equation}\label{eq:xi-estimate}
\left| L(\Diff) |\xi|^2 \right|_{L^2} \eqsim \left| |\xi|^2 \right|_{L^2} = \left| \xi \right|_{L^4}^2 \lesssim \left| \xi \right|_{H^1}^2.
\end{equation}
Using the embedding \(H^1(\R^2) \hookrightarrow L^6(\R^2)\), we furthermore find that
\begin{align}
\left| |\rho_1|^2-|\rho_2|^2\right|_{L^2}^2 & = \sum_{j\in \Z^2}\left|\mathrm{Re}\left( (\rho_1-\rho_2)
\overline{(\rho_1+\rho_2)}\right)\right|_{L^2(Q_j)}^2 \nonumber \\
& \leq \sum_{j\in \Z^2}\Big\langle |\rho_1-\rho_2|,
|\rho_1-\rho_2| |\rho_1+\rho_2|^2\Big\rangle_{L^2(Q_j)} \nonumber \\
& \leq \sum_{j\in \Z^2}
\left| \rho_1-\rho_2\right|_{L^2(Q_j)}
\big| |\rho_1|+|\rho_2| \big|_{L^6(Q_j)}^3\nonumber \\
& \leq \sup_{j \in \Z^2}  \left| \rho_1-\rho_2\right|_{L^2(Q_j)}
\sum_{j\in \Z^2} \big| |\rho_1|+|\rho_2| \big|_{H^1(Q_j)}^3\nonumber \\
& \leq \sup_{j \in \Z^2}  \left| \rho_1-\rho_2\right|_{L^2(Q_j)}
\big| |\rho_1|+|\rho_2| \big|_{H^1(\R^2)} \sum_{j\in \Z^2} \big| |\rho_1|+|\rho_2| \big|_{H^1(Q_j)}^2\nonumber \\
& = \sup_{j \in \Z^2}  \left| \rho_1-\rho_2\right|_{L^2(Q_j)}
\big| |\rho_1|+|\rho_2| \big|_{H^1(\R^2)}^3. \label{eq:rho-estimate}
\end{align}
Combining the Cauchy-Schwarz inequality with \eqref{eq:xi-estimate} and \eqref{eq:rho-estimate} yields the  proposition.
 \end{proof}

\begin{corollary} \label{cor:no vanishing}
Any sequence $\{\zeta_n\} \subset  B_M(0)$ such that ${\mathcal S}(\zeta_n) \gtrsim 1$
satisfies
\[
\sup_{j \in \Z^2}  |\zeta_n|_{L^2(Q_j)} \gtrsim 1.
\]
\end{corollary}

\begin{proof}
Using Proposition \ref{prop:1/2} with \(\xi = \varrho_1 = \zeta_n\) and \(\varrho_2 = 0\), one finds that
\[
 {\mathcal S}(\zeta_n)= \big| \sqrt{L(\Diff)}|\zeta_n|^2 \big|_{L^2}^2 \lesssim
\sup_{j \in \Z^2} |\zeta_n|_{L^2(Q_j)}^\frac{1}{2} |\zeta_n|_{H^1}^{\frac{7}{2}} \lesssim
\sup_{j \in \Z^2} |\zeta_n|_{L^2(Q_j)}^\frac{1}{2}.\qedhere
\]
\end{proof}

Altogether we have established the following existence result for the DS and FDKP functionals.

\begin{theorem}\label{thm:first existence theorem}
There exists a minimising sequence  \(\{\zeta_n\} \subset H_\varepsilon^1(\R^2)\) for $\mathcal{T}_\varepsilon\vert_{N_\varepsilon}$ with the properties that
\begin{itemize}
\item[(i)] \(\{\zeta_n\}\) converges weakly in \(H_\varepsilon^1(\R^2)\)  to a critical point \(\zeta_\infty\) of the DS functional \(\mathcal{T}_\varepsilon\) for $\varepsilon \geq 0$,
\item[(ii)] the corresponding sequence \(\{u(\zeta_n)\}\)
converges weakly in $X$ to a critical point $u_\infty = u(\zeta_\infty)$ of the FDKP functional \(I_\varepsilon\) for $\varepsilon>0$.
\end{itemize}
\end{theorem}

\section{Ground states} \label{sec:ground states}

In this section we strengthen Theorem~\ref{thm:first existence theorem} by showing that we
can choose the translational sequence $\{w_n\}$ appearing in Lemma~\ref{lemma:critical points}(i) to ensure strong convergence of (a subsequence of) the Palais-Smale sequence
$\{\zeta_n(\cdot+w_n)\}$
 in \(H^1_\varepsilon(\R^2)\) to a ground state. This observation will also provide us with some additional convergence results in the limit \(\varepsilon \to 0\). For these purposes we use an abstract concentration-compactness result by Ehrnstr\"{o}m \& Groves \cite[Thm 5.1]{EhrnstroemGroves18}, noting that any minimising sequence $\{\zeta_n\}$ of ${\mathcal T}_\varepsilon|_{N_\varepsilon}$
satisfies
$$\sup_{j \in \Z^2} |\zeta_n|_{L^2(Q_j)} \gtrsim 1$$
(because of Remark \ref{rem:inf is positive}, Corollary~\ref{cor:limsup < M-1}, Corollary~\ref{cor:no vanishing} and Lemma~\ref{lemma:general convergence}).

\begin{theorem}\label{thm:cc}
Let $H_1 \hookrightarrow H_0$ be Hilbert spaces, and
consider a sequence $\{x_n\} \subset l^2(\Z^s,H_1)$, where $s \in \N$.
Writing $x_n=(x_{n,j})_{j\in \Z^s}$, where  $x_{n,j}\in H_1$, suppose that\\[-10pt]
\begin{itemize}
\item[(i)] $\{x_n\}$ is bounded in $l^2(\Z^s,H_1)$,\\[-8pt]
\item[(ii)] $S=\{x_{n,j}:n \in \N, j \in \Z^s\}$ is relatively compact in $H_0$,\\[-8pt]
\item[(iii)] $\limsup_{n\to \infty}|x_n|_{l^\infty(\Z^s,H_0)} { \gtrsim 1}$.\\[-10pt]
\end{itemize}
For each $\Delta>0$ the sequence $\{x_n\}$ admits a subsequence
with the following properties.
{ There exist a finite number $m$ of non-zero vectors $x^1,\ldots,x^m\in l^2(\Z^s,H_1)$ and
sequences $\{w^1_n\}$, \ldots, $\{w^m_n\}
 \subset \Z^s$ satisfying
\[
\lim_{n \to \infty} |w_n^{m^{\prime\prime}}-w_n^{m^\prime}| \to \infty, \qquad 1 \leq m^{\prime\prime} < m^\prime \leq m
\]
such that
\begin{align*}
T_{-w^{m^\prime}_n}x_n &\rightharpoonup x^{m^\prime},\\
|x^{m^\prime}|_{l^\infty(\Z^s,H_0)} &=
\lim_{n\to \infty}
\left|x_n-\sum_{l=1}^{m^\prime-1}T_{w^l_n}x^l\right|_{l^\infty(\Z^s,H_0)}, \\
 \lim_{n\to \infty}|x_n|_{l^2(\Z^s,H_1)}^2 &=
\sum_{l=1}^{m^\prime} |x^l|_{l^2(\Z^s,H_1)}^2+
\lim_{n\to \infty}
\left|x_n-\sum_{l=1}^{m^\prime}T_{w^l_n}x^l\right|_{l^2(\Z^s,H_1)}^2
\end{align*}
for $m^\prime = 1, \ldots, m$,
\[
\limsup_{n\to \infty}\left|
x_n-\sum_{l=1}^mT_{w^l_n}x^l
\right|_{l^\infty(\Z^s,H_0)}\leq \Delta,
\]
and
\[
\lim_{n\to \infty}\left|x_n-T_{w^1_n}x^1
\right|_{l^\infty(\Z^s,H_0)}=0 
\]
if $m=1$. 
Here the weak convergence is understood in $l^2(\Z^s,H_1)$ and
$T_w$ denotes the translation operator $T_w (x_{n,j})=(x_{n,j-w})$.}
\end{theorem}

We proceed by using Theorem~\ref{thm:cc} to study { Palais--Smale} sequences for
${\mathcal T}_\varepsilon$, extracting subsequences where necessary for the
validity of our arguments.

\begin{lemma} \label{lemma:application of cc}
Suppose that $\{\zeta_n\} \subset B_{M-1}(0)$ satisfies
\[\lim_{n \to \infty} \norm{\diff {\mathcal T}_\varepsilon[\zeta_n]}_{\tilde{Y}\to \mathbb{R}}=0,
\qquad
\sup_{j \in \Z^2} |\zeta_n|_{L^2(Q_j)} \gtrsim 1.\]
There exists $\{w_n\} \subset \Z^2$ and a nontrivial critical point $\zeta_\infty$ of ${\mathcal T}_\varepsilon$
such that $\zeta_n(\cdot+w_n)
\rightharpoonup \zeta_\infty$ in $H^1(\R^2)$, ${\mathcal S}(\zeta_n) \to
{\mathcal S}(\zeta_\infty)$ as $n \to \infty$ and
\[
\lim_{n \to \infty} \sup_{j \in \Z^2} |\zeta_n(\cdot+w_n)-\zeta_\infty|_{L^2(Q_j)}=0.
\]
\end{lemma}

\begin{proof} 
Set
$H_1=H^1(Q_0)$,
$H_0=L^2({Q_0})$,
define $x_n\in l^2(\Z^2,{ H_1})$ for $n \in \N$ by
\[
x_{n,j}=\zeta_n(\cdot+j)|_{{Q_0}}
\in H^1(Q_0), \qquad j\in \Z^2,
\]
and apply Theorem~\ref{thm:cc} to the sequence $\{x_n\}\subset l^2(\Z^2,H_1)$,
noting that
\[
|x_n|_{l^2(\Z^2,H_1)}=|\zeta_n|_{H^1}, \qquad
|x_n|_{l^\infty(\Z^2,H_0)}=\sup_{j\in\Z^2}|
\zeta_n
|_{L^2(Q_j)}
\]
for $n \in \N$. Assumption (ii) is satisfied because $H^1(Q_0)$ is compactly embedded
in $L^2({Q_0})$, while assumptions (i) and (iii) follow from the hypotheses in the lemma.

The theorem
asserts the existence of a natural number $m$, sequences
$\{w_n^1\}, \ldots, \{w_n^m\} \subset \Z^2$ with
\begin{equation}
\lim_{n \to \infty} |w_n^{m^{\prime\prime}}-w_n^{m^\prime}| { =} \infty, \qquad 1 \leq m^{\prime\prime} < m^\prime \leq m,
\label{ Split}
\end{equation}
and functions
$\zeta^1,\ldots,\zeta^m\in B_M(0)\setminus\{0\}$ such that
$\zeta_n(\cdot+w^{m^\prime}_n) \rightharpoonup \zeta^{m^\prime}$ in $H^1(\R^2)$ as $n \to \infty$,
\begin{equation}\label{eq:abstract convergence}
\limsup_{n\to \infty}\sup_{j\in\Z^2}\left|
\zeta_n-\sum_{l=1}^m\zeta^l(\cdot-w^l_n)
\right|_{L^2(Q_j)}
\leq \varepsilon,
\end{equation}
\[
\sum_{l=1}^m|\zeta^l|_{H^1}^2\leq
\limsup_{n\to \infty}|\zeta_n|_{H^1}^2
\]
and
\begin{equation}
\label{eq:concentrate}
\lim_{n\to \infty}\sup_{j\in\Z^2}\left|\zeta_n-\zeta^1(\cdot-w^1_n)
\right|_{L^2(Q_j)}
=0
\end{equation}
if $m=1$. It follows from Lemma \ref{lemma:critical points}(i)
that $\mathrm{d}{\mathcal T}_\varepsilon[\zeta^l]=0$, so that
$\zeta^l \in N_\varepsilon$ and ${\mathcal T}_\varepsilon(\zeta^l) \geq \tau_\varepsilon  \gtrsim 1$ in view of Remark~\ref{rem:inf is positive}. Define
\[\tilde{\zeta}_n=\sum_{l=1}^m\zeta^l(\cdot-w^l_n), \qquad n \in \N,\]
and note that
\begin{align*}
\lim_{n\rightarrow \infty }|\zeta^{\ell_i}(\cdot-w_n^{\ell_i})
\overline{\zeta^{\ell_j}(\cdot-w_n^{\ell_j})}|_{L^2}^2 &= \lim_{n\rightarrow \infty }\Big\langle 
|\zeta^{\ell_i}(\cdot-w_n^{\ell_i})|^2, |\zeta^{\ell_j}(\cdot-w_n^{\ell_j})|^2\Big\rangle
\\
&= \lim_{n\rightarrow \infty }\Big\langle 
\ee^{-\ii w_n^{\ell_i} \cdot k}\widehat{|\zeta^{\ell_i}|^2},
\ee^{-\ii w_n^{\ell_j} \cdot k}\widehat{|\zeta^{\ell_j}|^2}\Big\rangle =0
\end{align*}
for $\ell_i\neq\ell_j$ by the Riemann--Lebesgue lemma (\(\{w^{\ell_i}\}\) diverges by \eqref{ Split}). Since introducing the Fourier multiplier \(L(\Diff)\) in the inner product does not
change this calculation, we find that
\begin{align}
&\!\!\!\!\lim_{n\rightarrow \infty } {\mathcal S}(\tilde{\zeta}_n) \nonumber \\
&=\lim_{n\rightarrow \infty }\sum_{\ell_1,\ell_2,\ell_3,\ell_4} \left\langle  L(\Diff) \big( \zeta^{\ell_1}(\cdot-w_n^{\ell_1})
\overline{\zeta^{\ell_2}(\cdot-w_n^{\ell_2})} \big),
\zeta^{\ell_3}(\cdot-w_n^{\ell_3})\overline{\zeta^{\ell_4}(\cdot-w_n^{\ell_4})}\right\rangle
\nonumber \\
&=  \lim_{n\rightarrow \infty } \sum_{\ell_1, \ell_3} \Big\langle L(\Diff)
|\zeta^{\ell_1}(\cdot-w_n^{\ell_1})|^2, |\zeta^{\ell_3}(\cdot-w_n^{\ell_3})|^2\Big\rangle \nonumber \\
&=\lim_{n\rightarrow \infty }\sum_{\ell=1}^m \big| \sqrt{L(\Diff)}|\zeta^{\ell}|^2 \big|_{L^2}^2 \nonumber\\
&=\sum_{\ell=1}^m {\mathcal S}(\zeta^\ell). \label{eq:first S estimate}
\end{align}

From Proposition \ref{prop:1/2} and equation \eqref{eq:abstract convergence}, one finds that
\begin{align}
& \hspace{-5mm}\limsup_{n \to \infty} |{\mathcal S}(\zeta_n) - {\mathcal S}(\tilde{\zeta}_n)| \nonumber \\
 & = \limsup_{n \to \infty}  \big\langle L(\Diff) ( |\zeta_n|^2-|\tilde{\zeta}_n|^2 ),|\zeta_n|^2+|\tilde{\zeta}_n|^2 \big\rangle \nonumber \\
 & \lesssim  { \limsup_{n \to \infty} \sup_{j \in \Z^2}} |\zeta_n-\tilde{\zeta}_n|_{L^2(Q_j)}^\frac{1}{2} 
(|\zeta_n|_{H^1}+|\tilde{\zeta}_n|_{H^1})^{\frac{7}{2}} \label{eq:second S estimate 1} \\
 & \leq \varepsilon^{\frac{1}{2}} \limsup_{n \to \infty} (|\zeta_n|_{H^1}+|\tilde{\zeta}_n|_{H^1})^{\frac{7}{2}}
 \nonumber \\
 & \lesssim  \varepsilon^{\frac{1}{2}} \label{eq:second S estimate 2}
\end{align}
uniformly in $m$. Because \(\zeta^l \in N_\varepsilon\), we may 
combine \eqref{eq:first S estimate} and \eqref{eq:second S estimate 2} with
\[
{\mathcal S}(\zeta^l) \geq \tau_\varepsilon-\bigO(\varepsilon^{\frac{1}{2}}), \qquad l=1,\ldots,m,
\]
(see Remark \ref{rem:inf is positive}) to obtain
\[
 \liminf_{n\to \infty} {\mathcal S}(\zeta_n)
\geq m \tau_\varepsilon-\bigO(\varepsilon^{\frac{1}{2}})
\]
and hence
\[
\tau_\varepsilon\geq m \tau_\varepsilon-\bigO(\varepsilon^{\frac{1}{2}})
\]
uniformly in $m$, because of \eqref{eq:only S}.  Since $\liminf_{\varepsilon\to 0} \tau_\varepsilon { \gtrsim 1}$ we deduce that $m=1$.
The desired result follows from
\eqref{eq:concentrate} with $\zeta_\infty=\zeta^1$ and $w_n=w_n^1$ and
\eqref{eq:second S estimate 1} since
${\mathcal S}(\tilde{\zeta}_n) \rightarrow {\mathcal S}(\zeta^1)$ as $n \rightarrow \infty$;
according to Lemma~\ref{lemma:critical points} the sequence $\{\zeta_n(\cdot+w_n)\}$ converges weakly to a 
nontrivial critical point of \(\mathcal{T}_\varepsilon\), and by uniqueness of limits we conclude that \(\diff \mathcal{T}_\varepsilon[\zeta_\infty]\)  vanishes. 
\end{proof}

We can now strengthen Theorem \ref{thm:first existence theorem}, dealing with the cases
$\varepsilon=0$ and $\varepsilon>0$ separately.

\begin{lemma} \label{lemma:general convergence}
Suppose that $\{\zeta_n\} \subset B_{M-1}(0)$ satisfies
\[
\lim_{n\rightarrow \infty}|\diff {\mathcal T}_0[\zeta_n]|_{H^1 \to \R} = 0,
\qquad
\sup_{j \in \Z^2} |\zeta_n|_{L^2(Q_j)} \gtrsim 1.\]
There exists $\{w_n\} \subset \Z^2$ such that
$\{\zeta_n(\cdot+w_n)\}$ converges strongly in $H^1(\R^2)$ to a nontrivial critical point of ${\mathcal T}_0$.
\end{lemma}

\begin{proof} 
Lemma \ref{lemma:application of cc} asserts the existence of 
$\{w_n\} \subset \Z^2$ and a nontrivial critical point $\zeta_\infty$ of ${\mathcal T}_0$ such that
$\zeta_n(\cdot+w_n)
\rightharpoonup \zeta_\infty$ in $H^1(\R^2)$ and  ${\mathcal S}(\zeta_n) \to
{\mathcal S}(\zeta_\infty)$ as $n \to \infty$. Abbreviating $\{\zeta_n(\cdot+w_n)\}$
to $\{\zeta_n\}$, we find from \eqref{eq:dT} that
\[{\mathcal Q}(\zeta_n) = \tfrac{1}{2}\diff {\mathcal T}_0[\zeta_n](\zeta_n) + 2{\mathcal S}(\zeta_n)
\to 2{\mathcal S}(\zeta_\infty) = {\mathcal Q}(\zeta_\infty)\]
as $n \to \infty$. Since \({\mathcal Q}(\zeta) \eqsim |\zeta|_{H^1}^2\) and \(\diff {\mathcal T}_0[\zeta_\infty] = 0\), it follows that $\zeta_n \to \zeta_\infty$ in $H^1(\R^2)$.
\end{proof}

\begin{theorem} \label{thm:second existence result, epsilon zero}
Let $\{\zeta_n\}$ be a minimising sequence for ${\mathcal T}_0|_{N_0}$
with
\[
\lim_{n\rightarrow \infty}|\diff {\mathcal T}_0[\zeta_n]|_{H^1 \to \R} = 0.
\] 
There exists $\{w_n\} \subset \Z^2$ such that $\{\zeta_n(\cdot+w_n)\}$
converges strongly in $H^1(\R^2)$ to a ground state of ${\mathcal T}_0$.
\end{theorem}

Let us now turn to the case $\varepsilon>0$, for which we need the following
technical result.

\begin{proposition}\label{prop:unbounded sequences}
Suppose that $u_n \rightharpoonup u_\infty$ in { $H^s(\R^2)$} as $n \to \infty$. The limit
\[
\lim_{n\to \infty} |u_n-u_\infty|_{L^\infty} = 0
\]
holds if and only if $u_n(\cdot - j_n) \rightharpoonup 0$ in $H^s(\R^2)$ as $n \to \infty$ for all
unbounded sequences $\{j_n\} \subset \Z^2$.
\end{proposition}

\begin{theorem}\label{thm:second existence result, epsilon positive}
{ Let $\varepsilon > 0$ and} $\{\zeta_n\}$ be a minimising sequence for ${\mathcal T}_\varepsilon|_{N_\varepsilon}$
with
\[
\lim_{n \to \infty}|\diff {\mathcal T}_\varepsilon[\zeta_n]|_{H^1 \to \R}=0.
\] 
There exists $\{w_n\} \subset \Z^2$ such that $\{\zeta_n(\cdot+w_n)\}$
converges weakly in $H^1_\varepsilon(\R^2)$ and strongly in \(L^\infty(\R^2)\) to a ground state $\zeta_\infty$ of ${\mathcal T}_\varepsilon$.
The corresponding FDKP sequence  
\[
\qquad  u_n = u_1(\zeta_n) + u_2(u_1(\zeta_n)), \qquad  \zeta_n = \zeta_n(\cdot+w_n),
\] 
converges weakly in $X$  and strongly in \(L^\infty(\R^2)\) to a nontrivial critical point 
\[
u_\infty = u_1(\zeta_\infty) + u_2(u_1(\zeta_\infty))
\] 
of ${\mathcal I}_\varepsilon$.
\end{theorem}

\begin{proof} Lemma \ref{lemma:application of cc} asserts the existence of 
$\{w_n\} \subset \Z^2$ and a nontrivial critical point $\zeta_\infty$ of ${\mathcal T}_\varepsilon$ such that
$\zeta_n(\cdot+w_n)
\rightharpoonup \zeta_\infty$ in $H_\varepsilon^1(\R^2)$ and ${\mathcal S}(\zeta_n) \to
{\mathcal S}(\zeta_\infty)$ as $n \to \infty$ and
\[
\lim_{n \rightarrow \infty} \sup_{j \in \Z^2} |\zeta_n(\cdot+w_n)-\zeta_\infty|_{L^2(Q_j)} \to 0.
\]
Note however that we cannot proceed as in the case $\varepsilon=0$ by
deducing that ${\mathcal Q}(\zeta_n) \rightarrow
{\mathcal Q}(\zeta_\infty)$ and hence $\zeta_n \rightarrow \zeta_\infty$ in $H^1(\R^2)$ as
$n \rightarrow \infty$ from equation \eqref{eq:dT} since we have not established that
${\mathcal E}_\varepsilon(\zeta_n) \rightarrow {\mathcal E}_\varepsilon(\zeta_\infty)$ as $n \rightarrow \infty$. Instead we transfer this argument to the FDKP functional ${\mathcal I}_\varepsilon$
(which has no corresponding remainder term).

Because
\[
H_\varepsilon^r(\R^2)  \cong \chi_\varepsilon(\Diff) L^2(\R^2)
\] 
for all \(r \geq 0\) and \(H^s(Q_j) \hookrightarrow L^\infty(Q_j)\)
(uniformly over $j \in \Z^2$), we find that
\begin{align*}
 \lim_{n \to \infty} |\zeta_n-\zeta_\infty|_{L^\infty}^2 &
 =  \lim_{n \to \infty} \sup_{j \in \Z^2} |\zeta_n-\zeta_\infty|_{L^\infty(Q_j)}^2  \\
 & \leq   \lim_{n \to \infty} \sup_{j \in \Z^2} |\zeta_n-\zeta_\infty|_{H^s(Q_j)}^2 \\
 & \leq  \lim_{n \to \infty} \sup_{j \in \Z^2} |\zeta_n-\zeta_\infty|_{L^2(Q_j)} |\zeta_n-\zeta_\infty|_{H^{2s}(Q_j)} \\
  & \leq  \lim_{n \to \infty} \sup_{j \in \Z^2} |\zeta_n-\zeta_\infty|_{L^2(Q_j)} |\zeta_n-\zeta_\infty|_{H^{2s}} \\
& =0,
\end{align*}
where have again abbreviated $\{\zeta_n(\cdot+w_n)\}$ to $\{\zeta_n\}$,
and Proposition \ref{prop:unbounded sequences} shows that
$\zeta_n(\cdot - j_n) \rightharpoonup 0$ in $ H^s_\varepsilon(\R^2)$
and hence in $H^1_\varepsilon(\R^2)$ as $n \to \infty$ for all unbounded sequences $\{j_n\} \subset \Z^2$.
Using Proposition \ref{prop:weak trace back}(iii), one finds that $u_n(\cdot - j_n) \rightharpoonup 0$ in $X$ and hence in $ H^s(\R^2)$ for all unbounded sequences $\{j_n\} \subset \Z^2$, so that $u_n \to u_\infty$ in
$L^\infty(\R^2)$ as $n \to \infty$ by Proposition~\ref{prop:unbounded sequences}. It follows that $u_n \to u_\infty$ in
$L^3(\R^2)$ and in particular that
\[\int_{\R^2} u_n^3 \dx \dy \to  \int_{\R^2} u_\infty^3 \dx \dy\]
as $n \rightarrow \infty$. Since Proposition~\ref{prop:weak trace back}(i) and~(ii) guarantee that
$\diff {\mathcal I}_\varepsilon[u_\infty]=0$ and $\diff {\mathcal I}_\varepsilon[u_n](u_n) \to 0$
as $n \rightarrow \infty$, one finds from
\[
{\mathcal I}_\varepsilon(u)=\frac{1}{2}\diff {\mathcal I}_\varepsilon[u](u) { - \frac{1}{6}} \int u^3 \dx \dy
\]
that ${\mathcal I}_\varepsilon(\zeta_n) \to {\mathcal I}_\varepsilon(\zeta_\infty)$ as $n \to \infty$. Consequently, $\lim_{n \to \infty} {\mathcal T}_\varepsilon(\zeta_n) = {\mathcal T}_\varepsilon(\zeta_\infty) = \tau_\varepsilon$.
\end{proof}

Finally, we show that critical points of ${\mathcal T}_\varepsilon$ converge to critical points of ${\mathcal T}_0$
as $\varepsilon \to 0$. The first step is to establish the corresponding convergence result for the infima of
these functionals over their natural constraint sets.

\begin{lemma} \label{lem:infima converge}
Let $\{\varepsilon_n\}$ be a sequence with $\lim_{n \to \infty} \varepsilon_n=0$ and
let $\zeta^{\varepsilon_n}$ be a ground state of ${\mathcal T}_{\varepsilon_n}$.
\begin{itemize}
\item[(i)]
One has that $\lim_{\varepsilon_n \to 0} \tau_{\varepsilon_n} = \tau_0$.
\item[(ii)]
There exists
$\{w_n\} \subset \Z^2$ and a ground state $\zeta^\star$ of ${\mathcal T}_0$ such that
$\{\zeta^{\varepsilon_n}(\cdot+w_n)\}$ converges to  $\zeta^\star$ in $H^1(\R^2)$ as $n \to \infty$.
\end{itemize}
\end{lemma}
\begin{proof}
Because ${\mathcal E}_\varepsilon\circ\chi_\varepsilon(\Diff)$ and $\diff {\mathcal E}_\varepsilon\circ\chi_\varepsilon(\Diff)$
converge uniformly to zero over $B_{M-1}(0) \subset H^1(\R^2)$ as $\varepsilon \to 0$,
 we find that
\[{\mathcal T}_{\varepsilon_n}(\zeta^{\varepsilon_n})-{\mathcal T}_0(\zeta^{\varepsilon_n})=o(1),
\qquad
\diff{\mathcal T}_{\varepsilon_n}[\zeta^{\varepsilon_n}]-\diff{\mathcal T}_0[\zeta^{\varepsilon_n}]=o(1)
\]
as $n \to \infty$ and hence that
\[
\lim_{n \to \infty}|\diff{\mathcal T}_0[\zeta^{\varepsilon_n}]|_{H^1 \to \R} =0.
\]
Combining Corollary \ref{cor:no vanishing} with
$S(\zeta^{\varepsilon_n}) \geq \tau_{\varepsilon_n}-\bigO(\varepsilon_n^{\frac{1}{2}})$ and $\liminf_{\varepsilon \to 0} \tau_\varepsilon { \gtrsim 1}$
yields
\[\sup_{j \in \Z^2} |\zeta^{\varepsilon_n}|_{L^2(Q_j)} \gtrsim 1.\]
According to Lemma \ref{lemma:general convergence} there exists
$\{w_n\} \subset \Z^2$ and $\zeta^\star \in N_0$ such that $\diff{\mathcal T}_0[\zeta^\star]=0$ and
$ \zeta^{\varepsilon_n}(\cdot+w_n) \to \zeta^\star$ in $H^1(\R^2)$ as $n \to \infty$.
It follows that
\begin{eqnarray}
\tau_0 & \leq & {\mathcal T}_0(\zeta^\star) \nonumber \\
& = & \lim_{n \to \infty} {\mathcal T}_0(\zeta^{\varepsilon_n}) \nonumber \\
& = & \lim_{n \to \infty} 
\big({\mathcal T}_0(\zeta^{\varepsilon_n})-{\mathcal T}_{\varepsilon_n}(\zeta^{\varepsilon_n})\big)
+ \lim_{n \to \infty} 
\big({\mathcal T}_{\varepsilon_n}(\zeta^{\varepsilon_n})- \tau_{\varepsilon_n}\big)
+ \liminf_{n \to \infty} \tau_{\varepsilon_n} \nonumber \\
& = &  \liminf_{n \to \infty} \tau_{\varepsilon_n}. \label{eq:inf est}
\end{eqnarray}

Proposition \ref{prop:approximate N0} (with $\lambda_0=1$ and $\zeta_0=\zeta^0$)
asserts the existence of $\xi_n \in N_{\varepsilon_n}$
with $\xi_n \to \zeta^0$ in $H^1(\R^2)$ and hence ${\mathcal T}_0(\xi_n) \to
{\mathcal T}_0(\zeta^0)= \tau_0$ as $n \to \infty$.
Because ${\mathcal E}\circ \chi_\varepsilon(\Diff)$
converges uniformly to zero over $B_{M-1}(0) \subset H^1(\R^2)$ as
$\varepsilon \to 0$, one finds that
\[{\mathcal T}_{\varepsilon_n}(\xi_n)-{\mathcal T}_0(\xi_n) = o(1)\]
as $n \to \infty$, whence
\begin{align}
\limsup_{n \to \infty} \tau_{\varepsilon_n}
& \leq \limsup_{n \to \infty} {\mathcal T}_{\varepsilon_n}(\xi_n) \nonumber \\
& = \lim_{n \to \infty}\big({\mathcal T}_{\varepsilon_n}(\xi_n) - {\mathcal T}_0(\xi_n)\big)
+ \lim_{n \to \infty}\big({\mathcal T}_0(\xi_n) - \tau_0\big) + \tau_0 \nonumber \\
& = \tau_0. \label{eq:sup est}
\end{align}
The stated results follow from inequalities \eqref{eq:inf est} and \eqref{eq:sup est}.
\end{proof}

Finally, we record the corresponding result for FDKP solutions.

\begin{theorem} 
Let $\{\varepsilon_n\}$ be a sequence with $\lim_{n \to \infty} \varepsilon_n=0$
and let $u^{\varepsilon_n}$ be a critical point of ${\mathcal I}_{\varepsilon_n}$
with ${\mathcal I}_{\varepsilon_n}(u^{\varepsilon_n})=\varepsilon_n^2 \tau_{\varepsilon_n}$,
so that  $u^{\varepsilon_n}(\zeta^{\varepsilon_n})$ defines a ground state $\zeta^{\varepsilon_n}$ of ${\mathcal T}_\varepsilon$. There exists
$\{w_n\} \subset \Z^2$ and a ground state $\zeta^\star$ of ${\mathcal T}_0$ such that { a subsequence of
$\{\zeta^{\varepsilon_n}(\cdot+w_n)\}$ converges to} $\zeta^\star$ in $H^1(\R^2)$ as $n \to \infty$.
\end{theorem}

\normalsize
\bibliographystyle{siam}
\bibliography{mdg}
\end{document}